\documentclass{siamart190516}
\usepackage{amsmath,amssymb} 
\usepackage[utf8]{inputenc}
\usepackage{hyperref}
\usepackage{graphicx}
\usepackage{algorithm,algorithmic,comment}
\usepackage{mystylefile}
\newtheorem{remark}{Remark}[section]

\title{Monte Carlo Methods for Estimating the Diagonal of a Real Symmetric Matrix}
\author{Eric Hallman\thanks{Department of Mathematics, North Carolina State University, Raleigh, NC, 27695 ({erhallma@ncsu.edu}, ipsen@ncsu.edu, asaibab@ncsu.edu)} \and Ilse C.F.\ Ipsen\footnotemark[1] \and Arvind K.\ Saibaba\footnotemark[1]}
\newcommand{\V}[1]{\boldsymbol{#1}}
\newcommand{\M}[1]{\boldsymbol{#1}}
\newcommand{\B}[1]{\boldsymbol{#1}}
\newcommand{\R}{\mathbb{R}}
\newcommand{\D}[1]{\mathcal{D}\,(#1)}
\newcommand{\mb}[1]{\mathbb{#1}}
\DeclareMathOperator{\diag}{\mathrm{diag}}
\DeclareMathOperator{\intdim}{\mathrm{intdim}}
\DeclareMathOperator{\Var}{\mathrm{Var}}
\DeclareMathOperator{\rank}{\mathrm{rank}}
\DeclareMathOperator{\trace}{\mathrm{trace}}
\newcommand{\expect}{\mathbb{E}}
\newcommand{\prob}{\mathbb{P}}

\usepackage{xcolor}

\newcommand{\AKS}[1]{\textcolor{red}{AKS: #1}}

\newcommand{\matc}[1]{\boldsymbol{\mathcal{#1}}}

\begin{document}
\maketitle
\begin{abstract}
For real symmetric matrices that are accessible only through matrix vector products, we present Monte Carlo estimators for computing the  diagonal elements. Our probabilistic bounds for normwise absolute and relative errors apply to Monte Carlo estimators based on random  Rademacher, sparse Rademacher, normalized and unnormalized Gaussian 
vectors, and to vectors with bounded fourth moments. 
The novel use of matrix concentration inequalities in our proofs
represents a systematic model for future analyses.
Our bounds mostly do not depend on the matrix dimension, 
target different error measures than existing work, and
imply that the accuracy of the estimators 
increases with the  diagonal dominance of the matrix. An application to derivative-based global sensitivity metrics corroborates this, as do
numerical experiments on synthetic test matrices.
We recommend against the use in practice of sparse Rademacher 
vectors, which are the basis for many 
randomized sketching and sampling algorithms, because 
they tend to deliver barely a digit of accuracy even under large 
sampling amounts.

\end{abstract}


\begin{AM}
15A15, 65C05, 65F50, 60G50, 68W20
\end{AM}

\section{Introduction}
We compute the diagonal elements of symmetric matrices $\M{A} \in \R^{n \times n}$
with Monte Carlo estimators of the form 
\begin{equation*}
\widehat{\M{A}}=\frac{1}{N}\sum_{k=1}^N{\M{A}\V{z}_k\V{z}_k^\top}
\end{equation*}
where $\V{z}_k$ are independent random vectors. This approach is crucial
when the elements of $\M{A}$ are available only
implicitly, via matrix vector products.

Estimating the diagonal elements of a matrix is important in many areas of science and engineering:
In electronic structure calculations, one computes the diagonal
elements of a projector onto the smallest eigenvectors of a Hamiltonian matrix~\cite{bekas2007estimator}. 
In statistics, leverage scores for column subset selection can be computed from the diagonals of the projector onto the column space. In Bayesian inverse problems, the diagonal elements of the posterior covariance are computed with matrix-free estimators. Diagonal, or Jacobi preconditioners can accelerate the convergence of iterative linear solvers~\cite{wathen2015preconditioning}. More recently, diagonal estimators have been used to accelerate second order optimization techniques for machine learning~\cite{yao2020adahessian}. In network science, subgraph centrality measures
and ranks the importance of the network nodes based on the diagonal of a scaled exponential of the adjacency matrix. In sensitivity analysis, Monte Carlo diagonal estimators  efficiently compute the derivative-based global sensitivity metrics~\cite{constantine2017global,kucherenko2017}.

Diagonal estimation is related to trace estimation. Once the diagonal elements are known, the trace can be computed from their sum. Therefore, estimators for the diagonal
of a matrix can be easily adapted to  trace estimators. Monte Carlo methods were first proposed by Hutchinson~\cite{hutchinson1989stochastic}, and subsequently improved and expanded to different distributions \cite{avron2011randomized,cortinovis2021randomized,roosta2015improved}. Applications of trace estimators, reviewed  in~\cite{ubaru2018applications}, include estimating density of states, log determinants, and Schatten $p$-norms.

\paragraph{Literature review}
 To our knowledge, Monte Carlo diagonal estimators were first proposed by Bekas, Kokiopoulou, and Saad~\cite{bekas2007estimator}, and a sufficient condition was given for a Monte Carlo estimator to be unbiased. However, this paper identified that large offdiagonal entries can result in large relative errors and developed probing methods to mitigate the effects of the offdiagonal entries. This idea is further explored in the following works~\cite{laeuchli2016methods,kaperick2019diagonal}.

 We are aware of a recent paper~\cite{BN22} as the only other work to analyze the number of samples required for a relative $(\epsilon,\delta)$ estimator. 
 In contrast to \cite{BN22},  our proofs are the first to exploit matrix concentration inequalities to impose a systematic structure that can serve as a model for future analyses and allow us to analyze the normwise errors in a different norm. We analyze more general distributions such as random vectors with bounded fourth moments and sparse Rademacher vectors with a user-specified sparsity parameter, and---in contrast to \cite{BN22}---focus on un-normalized estimators. Most of our bounds do not show an explicit dependence on the matrix dimension which is desirable for large-scale problems.

\subsection{Contributions and overview}
After introducing notation, relevant concentration inequalities, and the setup
for our analysis (section~\ref{s_not}),
we derive normwise error bounds for Monte Carlo estimators based
on independent Rademacher vectors
(section~\ref{s_rademacher}), random vectors with bounded fourth moments and Gaussian vectors (section~\ref{s_gauss}); componentwise bounds
for Rademacher and Gaussian  vectors (section~\ref{s_comp});
and apply Monte Carlo estimators to
derivative-based global sensitivity metrics (section~\ref{s_appl}).
Numerical experiments (section~\ref{s_num}) 
illustrate the accuracy of the Monte Carlo estimators and the bounds. The novel and noteworthy features of contributions are:

\begin{enumerate}
\item Most of our bounds do not depend on the matrix dimension $n$, and hold for all
 symmetric matrices, whether positive definite or not.
 \item We extend the concept of relative $(\epsilon,\delta)$ estimators to diagonal
 estimation to determine the minimal number of samples $N$ for a user-specified
choice of relative error $\epsilon$ and failure probability $\delta$
 (Definitions~\ref{def:epsdelest}, \ref{def:compdelest}).
 
\item Our normwise bounds suggest that for Rademacher vectors, the Monte Carlo 
estimators are more accurate for matrices that are more strongly diagonally
dominant (Theorem \ref{thm:rademacher}).
In particular, the least number of samples required for the
Monte Carlo estimators to achieve a 
user-specified relative error decreases with increasing diagonal dominance of $\M{A}$ in the relative  sense (Corollaries \ref{c:rademacher}, \ref{c_fourth}).
\item For Rademacher vectors parameterized in terms of sparsity levels
(Definition~\ref{d_sparserad}),
we show that the Monte Carlo estimators lose accuracy with increasing sparsity (Theorem~\ref{thm:rademacher3}, Corollary~\ref{c:rademacher3}).
Numerical experiments (section~\ref{s_num}) confirm that,
even for large sampling amounts, the estimators barely achieve
a single digit of accuracy. Therefore we recommend against their use
in practice.
\item Our componentwise bounds suggest that 
the accuracy for computing a diagonal element $a_{ii}$ depends only on the diagonal
dominance of column/row $i$ of $\M{A}$
(Corollaries~\ref{c:radcomp}, \ref{c:gauss:component}).
\item In the context of derivative-based global sensitivity metrics, we
design and analyze Monte Carlo estimators
based on random vectors from a
problem-specific probability distribution (Theorem~\ref{t_66}, Corollary~\ref{c_66}).
\end{enumerate}

\section{Background}\label{s_back}
After reviewing notation (section~\ref{s_not}) and
relevant concentration inequalities (section~\ref{s_rmt}), we present
the setup for our analysis (section~\ref{sec:analysis}).

\subsection{Notation}\label{s_not}
The \textit{Schur product} (or 
\textit{Hadamard}, or 
\textit{elementwise product})
of $\M{A},\M{B}\in \R^{m\times n}$ is denoted by 
$\M{C} = \M{A}\circ \M{B} \in \R^{m\times n}$ and has elements
\[c_{ij} = a_{ij}b_{ij} \qquad 1\leq 1 \leq m, \quad 1\leq j\leq n. \]
For $\M{A},\M{B},\M{C}\in\rmn$, the Schur product is commutative and distributive, 
\begin{align*}
\M{A}\circ \M{B} = \M{B}\circ \M{A}, \qquad 
\M{A}\circ(\M{B}+\M{C}) = \M{A}\circ\M{B} + \M{A}\circ\M{C}.
\end{align*}
Following MATLAB convention, we define 
$\diag(\M{A})=
\begin{pmatrix}a_{11} & \cdots & a_{nn}\end{pmatrix}^\top\in\mb{R}^n$
as the column vector of diagonal elements of $\M{A}\in\rnn$.
The operator $\diag$ is overloaded,
and $\diag(\V{x})\in\real^{n\times n}$ represents a 
diagonal matrix whose
diagonal elements are the elements of the vector $\V{x}\in\real^n$. 
In particular, 
\begin{align}\label{e_diag}
\D{\M{A}} \equiv \diag(\diag(\M{A}))=
\M{I} \circ \M{A}\ \in\real^{n\times n}
\end{align}
represents the diagonal matrix whose diagonal 
elements are the diagonal elements of $\M{A}$.
In other words, $\M{I} \circ \M{A}$ zeros out the offdiagonal
elements of $\M{A}$.

If the first factor in a Schur product is a square matrix
$\M{M}\in\R^{n\times n}$,  and the second factor
an outer product involving $\V{x}, \V{y} \in \R^n$, then
\begin{align}\label{e_outer}
\M{M}\circ (\V{xy}^\top) = \diag(\V{x})\,\M{M}\,\diag(\V{y}).
\end{align}
For symmetric matrices $\M{A}, \M{B} \in \R^{n\times n}$,
the partial order $\M{A} \preceq \M{B}$,
or equivalently $\M{B}\succeq \M{A}$,
says that $\M{B}-\M{A}$ is positive semidefinite.
If $\M{A}$ and $\M{B}$ are positive semidefinite, then $\M{A} \preceq \M{B}$ implies $\M{A}^{1/2} \preceq \M{B}^{1/2}$. 

The \textit{intrinsic dimension} of a nonzero symmetric positive semidefinite matrix 
$\M{A} \in \R^{n\times n}$ is 
\[ \intdim(\M{A}) \equiv \frac{\trace(\M{A})}{\|\M{A}\|_2},
\qquad \text{with}\quad 1\leq \intdim(\M{A}) \leq \rank(\M{A}) \leq n.\] 
If, additionally, $\M{A}$ is a diagonal matrix, then 
\[ \intdim(\M{A}) = \frac{\sum_{i=1}^n a_{ii}}{\max_{1 \leq i \leq n} a_{ii}}.\]
 The columns of 
$\M{A}=\begin{bmatrix}\V{a}_1 & \cdots & \V{a}_n\end{bmatrix}\in\real^{m\times n}$ are $\V{a}_j\in\real^m$, $1\leq j\leq n$, and the columns of the identity 
$\M{I}=\begin{bmatrix} \V{e}_1 & \cdots & \V{e}_n\end{bmatrix}\in\real^{n\times n}$ are $\V{e}_j\in\real^n$.
The transpose of $\M{A}$ is $\M{A}^{\top}$.

\subsection{Concentration inequalities}\label{s_rmt}
We rely on two scalar and two matrix concentration inequalities.

Markov's inequality \cite[Section 3.1]{Mitz2005} bounds the
the probability that a random variable exceeds a constant.

\begin{theorem}[Markov's inequality]\label{t_Markov}
If $Z$ is a non-negative random variable, then for $t >0$ 
\begin{equation*}
\prob[Z \geq t] \leq \frac{\expect[Z^2]}{t^2}.
\end{equation*}
\end{theorem}

Hoeffding's inequality for general bounded random variables 
\cite[Theorem 2.2.6]{vershynin2018high} bounds the
probability that a sum of scalar random variables exceeds its mean.

\begin{theorem}[Scalar Hoeffding inequality]\label{t_Hoeffding}
Let $Z_1, \ldots, Z_N$ be independent random variables, bounded by 
$m_k\leq Z_k\leq M_k$, $1\leq k\leq N$, with 
sum $Z\equiv \sum_{k=1}^N{Z_k}$. Then for $t>0$
\[ \myP\left[|Z-\E[Z]|\geq t \right] \leq 2
\exp\left(\frac{-2t^2}{\sum_{k=1}^N{(M_k-m_k)^2}} \right). \]
\end{theorem}

Next are two bounds 
for sums of independent symmetric matrix-valued random variables.
The first is a matrix Bernstein concentration inequality 
\cite[Theorems 7.3.1 and~7.7.1]{tropp2015introduction} 
for sums of independent, symmetric, bounded, zero-mean random matrices.

\begin{theorem}[Matrix Bernstein inequality]\label{thm:bernstein2}
Let $\M{S}_1,\dots,\M{S}_N \in \R^{n\times n}$ be independent
symmetric random matrices with
\[ \expect[\M{S}_k ] = \M{0}, \qquad \|\M{S}_k\|_2 \leq L \qquad  1 \leq k \leq N. \]
Let the sum $\M{S}\equiv \sum_{k=1}^N\M{S}_k$ have a matrix-valued variance that is majorized by 
$\M{V}\in\R^{n\times n}$,
\[ \M{V} \succeq \Var(\M{S}) = \expect[\M{S}^2] = \sum_{k=1}^N\expect[\M{S}_k^2].\]
Abbreviate $\nu \equiv \|\M{V}\|_2$ and $d \equiv \intdim(\M{V})$. 
Then for $t>0$
\begin{equation}\label{e_b}
\myP\left[\|\M{S}\|_2\geq t \right] \leq 
8d\, \exp\left(\frac{-t^2}{2(\nu + Lt/3)} \right). 
\end{equation}
\end{theorem}

\begin{proof}
In \cite[Theorems 7.3.1 and~7.7.1]{tropp2015introduction} it is shown that (\ref{e_b})
holds, provided $t \geq \sqrt{\nu} + \frac{L}{3}$.
We show that (\ref{e_b}) always holds and the lower bound on $t$ is not necessary.
To see this, note that
\[ \frac{-t^2}{2(\nu + Lt/3)} \]
decreases monotonically as $t$ increases. Therefore we can bound it from below 
as long as $t < \sqrt{\nu} + L/3$, by
\[
    \frac{-t^2}{2(\nu + Lt/3)} > 
    \frac{-(\sqrt{\nu} + L/3)^2}{2\left(\nu + (L/3)(\sqrt{\nu}+L/3)\right)} \geq -\frac{2}{3}.
\]
The second inequality comes from 
setting $x=\frac{L}{3\sqrt{\nu}}$ and noting that $f(x)=-\frac{(1+x)^2}{2(1+x+x^2)}$ has a minimum
at $\hat{x}=1$ where $f(\hat{x})=-2/3$, and $\sqrt{\nu} = L/3$. Substituting the lower bound into
Theorem \ref{thm:bernstein2} gives
\[8d\, \exp\left(\frac{-t^2}{2(\nu + Lt/3)} \right) \geq 8\exp(-2/3) > 4.\]
But now (\ref{e_b}) holds trivially since
\[\myP\left[\|\M{S}\|_2\geq t \right] \leq  1 < 4 < 
8d\, \exp\left(\frac{-t^2}{2(\nu + Lt/3)} \right).\] 
 
\end{proof}

The second matrix concentration inequality \cite[Theorem 3.2]{chen2012masked} 
bounds the mean of the squared norm of the sum of symmetric random matrices.

\begin{theorem}\label{thm:symmrand} Let $\M{S}_1,\dots,\M{S}_N \in \R^{n\times n}$
with $n\geq 3$ be independent symmetric random matrices with zero mean. Then 
\[\expect\left[\left\|\sum_{k=1}^N\M{S}_k\right\|_2^2\right]^{1/2} \leq \sqrt{2e\ln{n}}\> \left\| \left(\sum_{k=1}^N\expect[\M{S}_k^2]\right)^{1/2} \right\|_2 + 4e\ln{n}\> 
\left(\expect\left[\max_{1\leq k \leq N}{\|\M{S}_k\|_2^2}\right]\right)^{1/2}.\]
\end{theorem}

\subsection{Setup for the analysis}\label{sec:analysis}
Our Monte Carlo estimators 
compute the diagonal elements
of a symmetric 
matrix $\M{A} \in \R^{n\times n}$ 
by means of matrix vector products with $\M{A}$.
It samples $N$ independent random vectors $\V{w}_k\in\real^n$, 
and approximates the
vector of diagonal elements $\diag(\M{A})\in\real^n$
by the mean
\[ \diag(\widehat{\M{A})}= \frac{1}{N} \sum_{k=1}^N{\left((\M{Aw}_k)\circ \V{w}_k\right)}\in\real^n\qquad
\text{where}\quad \widehat{\M{A}}\equiv \frac{1}{N}\sum_{k=1}^N {\M{Aw}_k \V{w}_k^\top}\in\real^{n\times n}.\]
To see this, apply (\ref{e_diag}) and  (\ref{e_outer}) to
\begin{align*}
\D{\widehat{\M{A}}} = \M{I} \circ \widehat{\M{A}} &= \frac1N \sum_{k=1}^N{\diag(\M{Aw}_k)\M{I}\diag(\V{w}_k)} = \frac1N \sum_{k=1}^N{\diag((\M{Aw}_k) \circ \V{w}_k)}\\
&=\diag\left(\frac1N \sum_{k=1}^N{((\M{Aw}_k) \circ \V{w}_k)}\right)
=\diag\left(\diag(\widehat{\M{A}})\right)
\in\real^{n\times n}.
\end{align*}
Alternately, the diagonal elements of the estimators can be expressed as 
\begin{align*}
\widehat{\M{A}}_{ii} &= \frac{1}{N}\sum_{k=1}^N{(\M{A}\V{w}_k\V{w}_k^\top)_{ii}}
= \frac{1}{N}\sum_{k=1}^N{(\M{A}\V{w}_k)_i\,(\V{w}_k^\top)_{i}}\\
&= \frac{1}{N}\sum_{k=1}^N\left((\M{Aw}_k)\circ \V{w}_k\right)_i,\qquad 1\leq i\leq n.
\end{align*}

We measure the cost of a diagonal estimator by the number $N$
of samples. To assess the accuracy, we introduce 
a relative error in the form of 
normwise and componentwise
$(\epsilon,\delta)$ estimators, which extend the notion of  $(\epsilon,\delta)$ trace estimator from~\cite{avron2011randomized,roosta2015improved}.

\begin{definition}[Normwise $(\epsilon,\delta)$ diagonal estimator]\label{def:epsdelest}
Let $\M{A} \in \rnn$ be symmetric. 
Given user-specified parameters $0<\epsilon,\delta<1$, we say
that $\D{ \widehat{\M{A}}}$ is a
normwise $(\epsilon,\delta)$ estimator for the diagonal elements of $\M{A}$, if  
\[ \| \D{\M{A}} -\D{ \widehat{\M{A}}}\|_2 \leq \epsilon\, \|\D{\M{A}}\|_2\]
holds with probability at least $1-\delta$.
\end{definition}

In other words, for a user-specified failure probability $\delta>0$ and tolerance $\epsilon>0$, the normwise relative error of 
the diagonal estimator is, with 
probability at most $1-\delta$, at most $\epsilon$.
The two-norm in Definition~\ref{def:epsdelest} can be replaced by any matrix-$p$ norm, because 
the $p$-norm of a diagonal matrix $\M{D} \in \R^{n\times n}$ is 
$\|\M{D}\|_p = \max_{1 \leq i \leq n} {|d_{ii}|}$ for $p\geq 1$. Next we define a componentwise $(\epsilon,\delta)$ estimator.

\begin{definition}[Componentwise $(\epsilon,\delta)$ diagonal estimator]\label{def:compdelest}
Let $\M{A} \in \rnn$ be symmetric. 
Given user-specified parameters $0<\epsilon,\delta<1$ and diagonal 
element $a_{ii}\neq 0$ of $\M{A}$, we say
that $\widehat{a}_{ii} = ( \widehat{\M{A}})_{ii}$ is a
componentwise $(\epsilon,\delta)$ estimator for $a_{ii}$, if  
\[  |\widehat{a}_{ii} -a_{ii}| \leq \epsilon\, |a_{ii}|\]
holds with probability at least $1-\delta$.
\end{definition}

\section{Normwise bounds for Rademacher random vectors}\label{s_rademacher}
We present normwise bounds for Monte Carlo
estimators based on standard
(section~\ref{s_radenorm})
and on sparse Rademacher vectors
(section~\ref{s_4rade}).

\subsection{Standard Rademacher vectors}\label{s_radenorm}
After defining Rademacher vectors
(De\--finition~\ref{d_rade}) and
discussing their properties 
(Remarks \ref{r_rade} and~\ref{r_radediag}), 
we present a normwise absolute error bound (Theorem~\ref{thm:rademacher}),
and a bound on the minimal sampling amount 
that makes the Rademacher Monte Carlo estimator a
normwise $(\epsilon,\delta)$ diagonal estimator (Corollary~\ref{c:rademacher}).

\begin{definition}\label{d_rade}
A {\rm Rademacher random variable} takes on the
values $\pm 1$ with equal probability~1/2.
A {\rm Rademacher vector} is a random vector whose elements are independent
Rademacher random variables.
\end{definition}

Standard Rademacher vectors have the advantage of cheap matrix vector products and immediately recovering diagonal matrices.

\begin{remark}\label{r_rade}
The elements $w_j$ of a Rademacher vector $\V{w}$ have the following properties:
\begin{enumerate}
\item Zero mean: $\expect[w_j]=0$
\item Constant square: $w_{j}^2=1$
\item Independence: $\expect[w_jw_i]=0$ for $i\neq j$.
\end{enumerate}
\end{remark}

\begin{remark}\label{r_radediag}
Standard Rademacher vectors recover a diagonal matrix with a single sample, $N=1$. 

To see this, let $\M{A} = \D{\M{A}}\in\real^{n\times n}$ be diagonal, and $\V{w}\in\real^n$ a Rademacher vector. 
Remark~\ref{r_rade} implies that $\M{Aw}\V{w}^\top\in\real^{n\times n}$ has diagonal elements 
$a_{ii} w_i^2 = a_{ii}$, $1\leq i\leq n$.
\end{remark}

As a consequence, we can focus the analysis of standard Rademacher-based estimators on non-diagonal matrices.
The results below are special cases of
those for sparse Rademacher vectors in  section~\ref{s_4rade}.

\begin{theorem}\label{thm:rademacher}
Let $\M{A}\in \R^{n\times n}$ be non-diagonal symmetric, and
\[K_1 \equiv \|\D{\M{A}^2}-\D{\M{A}}^2\|_2,
\qquad K_2 \equiv \|\M{A}-\D{\M{A}}\|_{\infty}, \qquad  
d\equiv (\|\M{A}\|_F^2-\|\D{\M{A}}\|_F^2)/K_1. \]
If
$\widehat{\M{A}} \equiv\frac1N\sum_{k=1}^N \M{Aw}_k \V{w}_k^\top$ is a Monte Carlo estimator
with independent Rademacher vectors $\B{w}_k \in \R^n$, 
$1\leq k\leq N$, then
the probability that
the absolute error exceeds $t>0$ is at most
\[ \myP\left[\|\D{\M{A}} -\D{\widehat{\M{A}}}\|_2 \geq t \right] \leq 8d\, 
\exp\left(\frac{-Nt^2}{2(K_1 + tK_2/3)}\right). \]
\end{theorem}

\begin{proof} 
This is the special case $s=1$ of Theorem~\ref{thm:rademacher3}.
\end{proof}

The constants $K_1$ and $K_2$ represent the absolute deviation
of $\M{A}$ from diagonality, and more specifically the degree of diagonal dominance of $\M{A}$ in the absolute sense.
Theorem~\ref{thm:rademacher} implies that the Rademacher
estimator has a small absolute error when applied
to strongly diagonally dominant matrices. In other words,
the normwise absolute error in the Rademacher estimator 
decreases with increasing diagonal dominance of $\M{A}$ in the absolute sense.

We determine the least sampling amount 
required for the Monte Carlo estimator with Rademacher vectors to be a normwise $(\epsilon,\delta)$ diagonal estimator.

\begin{corollary}\label{c:rademacher}
Let $\M{A}\in \R^{n\times n}$ be non-diagonal symmetric. 
Let 
\begin{align*}
K_1 &\equiv \|\D{\M{A}^2}-\D{\M{A}}^2\|_2,\\
\Delta_1 &\equiv\frac{K_1}{\|\D{\M{A}}\|_2^2}, \qquad
\Delta_2\equiv
\frac{\|\M{A}-\D{\M{A}}\|_\infty}{\|\D{\M{A}}\|_{\infty}},\qquad
  d\equiv \frac{\|\M{A}\|_F^2-\|\D{\M{A}}\|_F^2}{K_1},
  \end{align*}
and let
$\widehat{\M{A}} \equiv\frac1N\sum_{k=1}^N \M{Aw}_k \V{w}_k^\top$ a Monte Carlo estimator
with independent Rademacher vectors $\B{w}_k \in \R^n$, $1\leq k\leq N$.
Pick $\epsilon>0$. For any $0<\delta<1$, if the sampling amount is at least
\begin{equation}\label{eqn:samp_rad}
N\geq \frac{\Delta_2}{3\epsilon^2} 
\left(2\epsilon +6\,\frac{\Delta_1}{\Delta_2}\right) \ln{(8d/\delta)},
 \end{equation} 
then $\|\D{\M{A}} -\D{\widehat{\M{A}}}\|_2 \leq \epsilon\, \|\D{\M{A}}\|_2$
 holds with probability at least $1-\delta$.
\end{corollary}

\begin{proof}
This is the special case $s=1$ of Corollary~\ref{c:rademacher3}.
\end{proof}

The constants $\Delta_1$ and $\Delta_2$ in
Corollary~\ref{c:rademacher} represent the 
respective relative counterparts of $K_1$ and 
$K_2$ in Theorem~\ref{thm:rademacher}: they
represent the relative deviation
of $\M{A}$ from diagonality, and more specifically the degree of diagonal dominance of $\M{A}$ in the relative sense.
Corollary~\ref{c:rademacher} implies that 
if $\M{A}$ is strongly diagonally dominant in the relative sense, then a small sampling
amount suffices to make
the Rademacher estimator a normwise $(\epsilon,\delta)$ diagonal 
estimator.  As with many
randomized sampling algorithms, the lower bound for $N$
is proportional to $1/\epsilon^2$.

\subsection{Sparse Rademacher vectors}\label{s_4rade}
For Rademacher vectors that are
parameterized in terms of sparsity
(Definition~\ref{d_sparserad}), we
derive a normwise absolute error bound (Theorem~\ref{thm:rademacher3}),
followed by the minimal sampling amount that 
makes the sparse Rademacher estimator a
normwise $(\epsilon,\delta)$ diagonal estimator
(Corollary~\ref{c:rademacher3}).

The random vectors in~\cite{achlioptas2003database} 
have elements that assume values from the discrete distribution $\{-\sqrt{3}\, ,0\, ,\sqrt{3}\}$ with respective probability $\left\{\frac16,\, \frac23,\, \frac16\right\}$.
This concept was extended in
\cite[(2)]{LHC06} to Rademacher vectors that are parameterized
in terms of a sparsity parameter~$s$.

\begin{definition}\label{d_sparserad}
A {\rm sparse Rademacher} random variable with parameter $s \geq 1$ takes the values $\{-\sqrt{s},\, 0,\, \sqrt{s}\}$ with probability $\{\frac{1}{2s},\, 1-\frac{1}{s},\, \frac{1}{2s}\}$ respectively.

A {\rm Sparse Rademacher vector} is a random vector whose elements are
independent sparse Rademacher random variables.
\end{definition}

The properties of sparse Rademacher vectors
are almost the same as those of the original Rademacher vectors in Remark~\ref{r_rade}.

\begin{remark}\label{r_sparserad}
The elements of a sparse Rademacher vector $\V{w}$ with parameter $s\geq 1$
have the following properties
\begin{enumerate}
\item Zero mean: $\E[w_j]=0$
\item Unit variance $\E[w_j^2]=1$
\item Independence: For $i\neq j$,
and integer $\ell\geq1$ 
\begin{equation*}
\expect[w_i^2w_j^2]=1,\qquad
\expect[w_i^{\ell}w_j]=
\expect[w_iw_j^{\ell}]=0.
\end{equation*}
\end{enumerate}
\end{remark}

The case $s=1$ corresponds to the original Rademacher vectors (Definition~\ref{d_rade}), while $s=3$ corresponds to
the choice in \cite{achlioptas2003database}.

Below is the extension of  Theorem~\ref{thm:rademacher}
to sparse Rademacher vectors with 
integer parameters $s$.

\begin{theorem}\label{thm:rademacher3}
Let $\M{A}\in \R^{n\times n}$ be non-diagonal symmetric, and 
and let
\begin{align*}
K_1(s)&\equiv\|\D{\M{A}^2}+(s-2)\D{\M{A}}^2\|_2,\qquad
K_2(s) \equiv \|s\M{A}-\D{\M{A}}\|_\infty,\\
d(s)& \equiv \frac{1}{K_1(s)}\|\M{A}\|_F^2+(s-2)\|\D{\M{A}}\|_F^2. 
\end{align*}
Let
$\widehat{\M{A}}\equiv \frac1N\sum_{j=1}^N {\M{Aw}_k \V{w}_k^\top}$ be a Monte Carlo
estimator with independent sparse Rademacher random vectors $\V{w}_k\in \R^n$,
with integer parameter $s\geq 1$.
Then the probability that
the absolute error exceeds $t>0$ is at most
\[ \myP\left[\|\D{\M{A}} -\D{\widehat{\M{A}}}\|_2 \geq t \right] \leq 8d(s)\, 
\exp\left(\frac{-Nt^2}{2(K_1(s) + tK_2(s)/3)}\right). \]
\end{theorem}

\begin{proof} 
Define the random diagonal matrices
\begin{align*}
\M{S}_k \equiv\frac1N \left(\M{I}\circ (\M{A}\V{w}_k\V{w}_k^\top) - \M{I}\circ\M{A}\right),
\qquad 1\leq k\leq N,
\end{align*}
and their sum 
\begin{align}\label{e_Z3}
\M{Z} \equiv \sum_{k=1}^N\M{S}_k = \M{I}\circ \widehat{\M{A}} - \M{I}\circ\M{A}
=\D{\M{A}}-\D{\widehat{\M{A}}}.
\end{align}
Before applying Corollary~\ref{thm:bernstein2}, we need
to verify the assumptions for the Bernstein inequality.

\paragraph{1. Expectation}
With (\ref{e_outer}), the first summand in $\M{S}_k$ equals 
\begin{align*}
\M{I}\circ (\M{A}\V{w}_k\V{w}_k^\top) = \diag(\M{Aw}_k) \M{I}\diag(\V{w}_k) = \diag(\M{Aw}_k)\diag(\V{w}_k),
\end{align*}
and with (\ref{e_diag}) this gives
\begin{align}\label{e_Sk3}
\M{S}_k =\frac1N \left(\diag(\M{Aw}_k)\diag(\V{w}_k)- \D{\M{A}}\right),
\qquad 1\leq k\leq N.
\end{align}
The linearity of expectation implies
\[ \expect[\M{S}_k ] = \frac1N 
\left(\underbrace{\expect[\diag(\M{Aw}_k)\diag(\V{w}_k)]}_{\M{T}} - \D{\M{A}}\right),\qquad 1\leq k\leq N.\]
Remark~\ref{r_sparserad} implies that element $(i,j)$ of
$\M{T}\equiv\expect[\diag(\M{Aw}_k)\diag(\V{w}_k)]$ equals
\begin{align}\label{e_t3}
t_{ij} = \expect[ (\M{Aw}_k)_i\,\M{I}_{ij}\,(\V{w}_k)_j] = \M{I}_{ij}a_{ij} 
\qquad 1 \leq i\leq n, \ 1 \leq j \leq n. 
\end{align}
Hence $\M{T} = \M{I}\circ\M{A}$, and 
$\expect[\M{S}_k]=\frac{1}{N}(\M{T}-\M{I}\circ\M{A})=\M{0}$, $1\leq k\leq N$.
Therefore, the random matrices have zero mean, $\expect[\M{S}_k] = \M{0}$; 
and so has their sum, $\expect[\M{Z}]=\M{0}$. 

\paragraph{2. Boundedness}
From (\ref{e_Sk3}) follows that the diagonal matrices $\M{S}_k$ have diagonal elements
\begin{align*}
(\M{S}_k)_{ii}&= \frac{1}{N}\left((\M{A}\V{w}_k)_i
(\V{w}_k)_i-a_{ii}\right)=
\frac{1}{N}\left((\V{w}_k)_i\,
\sum_{j=1}^n{a_{ij}(\V{w}_k)_j}-a_{ii}\right)\\
&=\frac{1}{N}\left(a_{ii} ((\V{w}_k)_i^2-1) + (\V{w}_k)_i\,
\sum_{j\neq i}{a_{ij}(\V{w}_k)_j}
\right), \qquad 1\leq i\leq n,\quad 1\leq k\leq N.
\end{align*}
Since $s\geq 1$ is an integer,
$|((\V{w}_k)_{ii}^2-1)a_{ii}|\leq (s-1)|a_{ii}|$ holds for any value of 
$(\V{w}_k)_i$, and 
\begin{align*}
|(\M{S}_k)_{ii}|\leq \frac{1}{N}\left((s-1)|a_{ii}|+s\sum_{j\neq i}{|a_{ij}|}\right),\qquad 1\leq i\leq n,\quad 1\leq k\leq N.
\end{align*}
Since the matrix infinity norm is absolute,
with $\|\M{S}_k\|_{\infty}=\|\,|\M{S}_k|\,\|_{\infty}$
we can bound the two-norms of the diagonal matrices $\M{S}_k$ by
\begin{align*}
\|\M{S}_k\|_2&\leq
\max_{1\leq i\leq n}{\frac1N\left(
(s-1)|a_{ii}|+s\sum_{j\neq i}{|a_{ij}|}\right)}
=\max_{1\leq i\leq n}{\frac1N\left(-|a_{ii}|+
s\sum_{j=1}^n{|a_{ij}|}\right)}\\
&=\frac{1}{N}\|s\M{A}-\D{\M{A}}\|_{\infty}
 = \frac{K_2(s)}{N}, \qquad 1\leq k\leq N.
\end{align*}
Set $L(s)\equiv K_2(s)/N$ where $K_2(s)>0$ since $\M{A}$ is not diagonal.

\paragraph{3. Variance}
In (\ref{e_Sk3}) abbreviate $\M{D}_{k} \equiv \diag(\M{Aw}_k)$, $\M{W}_{k} \equiv \diag(\V{w}_k)$, and $\M{D} \equiv \D{\M{A}}$, and write 
the summands in $\Var[\M{Z}] = \sum_{k=1}^N{\expect[\M{S}_k^2]}$ as 
\begin{align*}  
 \M{S}_k^2 = 
 \frac{1}{N^2}\left(\M{D}_k\M{W}_k-\M{D}\right)^2=
\frac{1}{N^2} \left((\M{D}_{k}\M{W}_k)^2 -2\M{D}\M{D}_{k}\M{W}_{k} + \M{D}^2 \right),\qquad 1\leq k\leq N,
\end{align*}
taking advantage of the fact that diagonal matrices commute.
The expectation is
\begin{align}\label{e_skex3}
 \expect\left[\M{S}_k^2\right] =  
\frac{1}{N^2} \left(\expect\left[(\M{D}_{k}\M{W}_k)^2\right] -2\M{D}\,\expect\left[\M{D}_{k}\M{W}_{k}\right] + \M{D}^2 \right),\qquad 
1\leq k\leq N.
\end{align}
To determine $\expect[\M{D}_k\M{W}_k]$, look at the individual  diagonal elements
and abbreviate $\V{u}=\V{w}_k$,
\begin{align}\label{e_dkwk3}
(\M{D}_k\M{W}_k)_{ii}={u}_i\,\V{e}_i^\top\M{Au} ={u}_i\,\sum_{j=1}^n{a_{ij}{u}_j}=a_{ii}{u}_i^2+\sum_{j\neq i}{a_{ij}{u}_i{u}_j},
\qquad 1\leq i\leq n.
\end{align}
Remark~\ref{r_sparserad} implies for 
the expectation
\begin{align*}
\expect\left[(\M{D}_k\M{W}_k)_{ii}\right] =a_{ii}\,\expect\left[{u}_i^2\right]+
\sum_{j\neq i}{a_{ij}\,\expect\left[{u}_i{u}_j\right]}=a_{ii}, \qquad 1\leq i\leq n.
\end{align*}
Thus $\expect\left[\M{D}_k\M{W}_k\right]=\M{D}$.
Insert this into~(\ref{e_skex3})
\begin{align}\label{e_skvar3}
 \expect\left[\M{S}_k^2\right] =  
\frac{1}{N^2} \left(\expect\left[(\M{D}_{k}\M{W}_k)^2\right] - \M{D}^2 \right),\qquad 
1\leq k\leq N.
\end{align}
To determine $\expect\left[(\M{D}_{k}\M{W}_k)^2\right]$,
start from the individual diagonal elements in (\ref{e_dkwk3}) 
\begin{align*}
(\M{D}_k\M{W}_k)_{ii}^2
&=\left(a_{ii}{u}_i^2+{u}_i\,\sum_{j\neq i}{a_{ij}{u}_j}\right)^2\\
&=\underbrace{a_{ii}^2{u}_i^4}_{\alpha}+
\underbrace{2a_{ii}{u}_i^3\,
\sum_{j\neq i}{a_{ij}{u}_j}}_{\beta}+
\underbrace{{u}_i^2
\left(\sum_{j\neq i}{a_{ij}{u}_j}\right)^2}_{\gamma}
\qquad 1\leq i\leq n.
\end{align*}
In the expectation
\begin{align}\label{e_abg}
\expect\left[(\M{D}_k\M{W}_k)_{ii}^2\right]=\expect[\alpha]+\expect[\beta]+\expect[\gamma]
\end{align}
we inspect each summand in turn.

Definition~\ref{d_sparserad} implies for the fourth moment
\begin{align}\label{e_alpha}
\expect[\alpha]=a_{ii}^2
\expect\left[{u}_i^4\right]=
a_{ii}^2
\left(\frac{1}{2s}\left(-\sqrt{s}\right)^4+\left(1-\frac{1}{s}\right)0+\frac{1}{2s}\left(\sqrt{s}\right)^4\right)=a_{ii}^2s.
\end{align}
Since ${u}_i$ is independent from
${u}_j$ for $j\neq i$
Remark~\ref{r_sparserad} implies
\begin{align}\label{e_beta}
\expect[\beta]=2a_{ii}\sum_{j\neq i}{a_{ij}\, \underbrace{\expect\left[{u}_i^3{u}_j\right]}_{=0}}=0.
\end{align}
Collect the independent elements of $\V{u}$, 
\begin{align*}
\gamma&={u}_i^2\left(\sum_{j\neq i}{a_{ij}{u}_j}\right)\left(\sum_{\ell\neq i}{a_{i\ell}{u}_{\ell}}\right)
={u}_i^2\left(\sum_{j\neq i}{a_{ij}{u}_j\left(a_{ij}{u}_j+\sum_{\ell\neq i,j}{a_{i\ell}{u}_{\ell}} \right)}\right)\\
&=\sum_{j\neq i}{a_{ij}^2{u}_j^2{u}_i^2}
+\sum_{j\neq i}{\sum_{\ell\neq i,j}{a_{ij}a_{i\ell}{u}_j{u}_{
\ell}{u}_i^2}}
\end{align*}
and apply Remark~\ref{r_sparserad},
\begin{align}\label{e_gamma}
\expect[\gamma]
&=\sum_{j\neq i}{a_{ij}^2\,\underbrace{\expect\left[{u}_j^2{u}_i^2\right]}_{=1}}
+\sum_{j\neq i}{\sum_{\ell\neq i,j}{a_{ij}a_{i\ell}\,\underbrace{\expect\left[{u}_j{u}_{
\ell}{u}_i^2\right]}_{=0}}}
=\sum_{j\neq i}{a_{ij}^2}
\end{align}
Substitute (\ref{e_alpha}), (\ref{e_beta}) and (\ref{e_gamma}) into (\ref{e_abg}). Then Remark~\ref{r_sparserad} and the
symmetry of~$\M{A}$ imply
\begin{align*}
\expect\left[(\M{D}_k\M{W}_k)_{ii}^2\right]=sa_{ii}^2+\sum_{j\neq i}{a_{ij}^2}=
\sum_{j=1}^n{a_{ij}^2}+(s-1)a_{ii}^2=
\|\V{a}_i\|_2^2+(s-1)|a_{ii}|^2.
\end{align*}
With $\D{\M{A}^2}= \diag\begin{pmatrix}(\|\V{a}_1\|_2^2 & \cdots &\|\V{a}_n\|_2^2\end{pmatrix}$
this gives for the whole diagonal matrix 
\begin{align*}
\expect[(\M{D}_k\M{W}_k)^2]=\D{\M{A}^2}
+(s-1)\M{D}^2, \qquad 1\leq k\leq N.
\end{align*}
At last, substitute this into the 
variance (\ref{e_skvar3})
\begin{align*}
 \expect\left[\M{S}_k^2\right] &=  
\frac{1}{N^2}\left(\D{\M{A}^2}+(s-1) \M{D}^2 -\M{D}^2\right)\\
&=
\frac{1}{N^2}\left(\D{\M{A}^2}
+(s-2)\M{D}^2\right),\qquad 1\leq k\leq N.
\end{align*}
Sum up the individual variances,
\begin{align*}
\M{V}(s)\equiv \Var[\M{Z}] &= \sum_{k=1}^N{\expect[\M{S}_k^2]} =\frac{1}{N^2}\,
 \sum_{k=1}^N{\left(\D{\M{A}^2}+(s-2)\M{D}^2\right)} \\
&= \frac{1}{N} \left(\D{\M{A}^2}+(s-2)\M{D}^2 \right).
 \end{align*}
Since $\M{V}(s)$ is diagonal, its norm is 
\[\nu(s) \equiv  \|\M{V}(s)\|_2 =\frac{1}{N} \| \D{\M{A}^2} +(s-2) \M{D}^2 \|_2 
=\frac{K_1(s)}{N},\]
where $K_1(s)>0$ since $\M{A}$ is not diagonal.
The intrinsic dimension of $\M{V}(s)$ is 
\begin{align*}
d(s)&\equiv \intdim(\M{V}) = \frac{\trace(\D{\M{A}^2} +(s-2)\M{D}^2)}{N\nu(s)}
= \frac{1}{K_1(s)} 
\sum_{j=1}^n{\left( \|\V{a}_j\|_2^2 +(s-2) |a_{jj}|^2\right)}\\
&=\frac{1}{K_1(s)}\|\M{A}\|_F^2+(s-2)\|\D{\M{A}}\|_F^2.
\end{align*}

\paragraph{4. Apply Theorem~\ref{thm:bernstein2}}
Substituting $L(s)=K_2(s)/N$ and $\nu(s)=K_1(s)/N$ into Theorem~\ref{thm:bernstein2} 
and remembering that the sum $\M{Z}$ in (\ref{e_Z3}) has zero mean gives
\[ 
\myP\left[\|\M{Z}\|_2\geq t\right]=
\myP\left[\|\D{\M{A}} -\D{\widehat{\M{A}}}\|_2 \geq t\right] \leq 
8d(s)\, \exp\left(\frac{-Nt^2}{2(K_1(s) + tK_2(s)/3)}\right). \]
\end{proof}

In the special case $s=1$ of standard
Rademacher vectors,
Theorem~\ref{thm:rademacher3}
reduces to Theorem~\ref{thm:rademacher}.
However, sparse Rademacher Monte Carlo estimators with $s>1$ do, in general, not recover a diagonal matrix with a single sample, $N=1$.

As $s$ increases, so do
the constants $K_1(s)$ and $K_2(s)$, and the upper bound on
$\myP\left[\|\D{\M{A}} -\D{\widehat{\M{A}}}\|_2 \geq t \right]$.
In other words, the sparser the vectors
$\V{w}_k$, the less accurate the Monte Carlo estimate $\D{\widehat{\M{A}}}$.

\begin{remark}[Non-integer sparsity levels]
The restriction to integers $s$ in 
Theorem~\ref{thm:rademacher3} is 
relevant only for $1<s<2$. More generally, Theorem~\ref{thm:rademacher3} 
holds for $s=1$, and any real number $s\geq 2$.
\end{remark}

The extension below of Corollary~\ref{c:rademacher} 
presents the minimal sampling amount that
makes the sparse Rademacher Monte Carlo estimator
a normwise $(\epsilon,\delta)$ diagonal estimator.

\begin{corollary}\label{c:rademacher3}
Let $\M{A}\in \R^{n\times n}$ be non-diagonal symmetric, and let
\begin{align*}
K_1(s)&\equiv \|\D{\M{A}^2}+(s-2)\D{\M{A}}^2\|_2,
\qquad
\Delta_1(s) \equiv\frac{K_1(s)}{\|\D{\M{A}}\|_2}\\
\Delta_2(s)&\equiv
\frac{\|s\M{A}-\D{\M{A}}\|_\infty}{\|\D{\M{A}}\|_{\infty}},\qquad
  d\equiv \frac{\|\M{A}\|_F^2+(s-2)\|\D{\M{A}}\|_F^2}{K_1(s)}.
  \end{align*}
Let $\widehat{\M{A}}\equiv \frac1N\sum_{j=1}^N {\M{Aw}_k \V{w}_k^\top}$ be a Monte Carlo
estimator with independent sparse Rademacher random vectors $\V{w}_k\in \R^n$,
with integer parameter $s\geq 1$. Pick $\epsilon>0$. For any $0<\delta<1$, if the sampling amount is at least
\begin{equation*}
N\geq \frac{\Delta_2(s)}{3\epsilon^2} 
\left(2\epsilon +6\,\frac{\Delta_1(s)}{\Delta_2(s)}\right) \ln{(8d(s)/\delta)},
 \end{equation*} 
then $\|\D{\M{A}} -\D{\widehat{\M{A}}}\|_2 \leq \epsilon\, \|\D{\M{A}}\|_2$
 holds with probability at least $1-\delta$.
\end{corollary}

\begin{proof}
Denote the bound
for the failure probability in Theorem~\ref{thm:rademacher3} by
\[ \delta\equiv 8 d(s) \,\exp\left(\frac{-Nt^2}{2(K_1(s) + tK_2(s)/3)}\right), \]
and solve it for $t$, 
\begin{align*}
t = \frac{K_2(s)}{3N}\ln{(8d(s)/\delta)} + \sqrt{\frac{K_2(s)^2}{9N^2} \ln^2(8d(s)/\delta) + \frac{2K_1(s)}{N} \ln{(8d(s)/\delta)}}. 
\end{align*}
We can restate Theorem~\ref{thm:rademacher} in terms of the failure
probability: With probability at most $\delta$, the normwise absolute
error exceeds $\|\D{\M{A}}-\D{\widehat{\M{A}}}\|_2\geq t$, where
\begin{align*}
t=\frac{K_2(s)}{3N}\ln{(8 d(s)/\delta)} + \sqrt{\frac{K_2(s)^2}{9N^2} \ln^2(8d(s)/\delta) + \frac{2K_1(s)}{N} \ln{(8d(s)/\delta)}}.
\end{align*}
And in terms of the success probability:
With probability at least $1-\delta$, the normwise absolute
error is bounded above by $\|\D{\M{A}}-\D{\widehat{\M{A}}}\|_2\leq t$.
Converting this absolute error into a relative one
requires $t \leq \epsilon \|\D{\M{A}}\|_2$, in other words,
\[ \frac{t}{\|\D{\M{A}}\|_2}=
\gamma + \sqrt{\gamma^2 + \beta \gamma} \leq \epsilon, \] 
where 
\begin{align*}
\gamma \equiv\frac{K_2(s)\ln{(8d(s)/\delta)}}{3N\|\D{\M{A}}\|_2},\qquad
\beta \equiv \frac{6K_1(s)}{K_2(s)\|\D{\M{A}}\|_2}
\end{align*}
Solving for $\gamma$ gives $\gamma \leq \frac{\epsilon^2}{2\epsilon + \beta}$,
which implies
\begin{align*}
N &\geq \frac{K_2(s)}{3\epsilon^2\,\|\D{\M{A}}\|_2} \left(2\epsilon + \frac{6K_1(s)}{K_2(s)\|\D{\M{A}}\|_2} \right) \ln{(8d(s)/\delta)}\\
&=\frac{1}{3\epsilon^2}\,\frac{K_2(s)}{\|\D{\M{A}}\|_2} \left(2\epsilon +6\,\frac{K_1(s)}{\|\D{\M{A}}\|_2^2}
\frac{\|\D{\M{A}}\|_2}{K_2(s)} \right) \ln{(8d(s)/\delta)} \\
& =\frac{\Delta_2}{3\epsilon^2} 
\left(2\epsilon +6\,\frac{\Delta_1(s)}{\Delta_2(s)}\right) \ln{(8d(s)/\delta)},
\end{align*} 
where we exploited the norms of diagonal matrices,
\begin{align*}
\Delta_1(s) &=\frac{K_1(s)}{\|\D{\M{A}}\|_2^2},\qquad
\Delta_2=\frac{\|\M{A}-\D{\M{A}}\|_\infty}{\|\D{\M{A}}\|_{\infty}}=
\frac{K_2(s)}{\|\D{\M{A}}\|_2}.
\end{align*}
\end{proof}

 Corollary~\ref{c:rademacher3}
suggests that increasing the sparsity parameter $s$
could on the one hand
lower the computational cost per sample, but on the 
other hand increase the sampling amount for the same accuracy.

\section{Gaussian vectors}\label{s_gauss}
We present normwise bounds for random
vectors with bounded fourth moment
(section~\ref{s_bounded}), and
standard Gaussian vectors
(section~\ref{s_4gauss}).

\subsection{Random vectors with bounded fourth moment}\label{s_bounded}
We bound the expectation of the squared absolute
error (Theorem~\ref{t_fourth}) for 
Monte Carlo estimators based on
random vectors $\V{w}_k$ (assumed to have independent entries with zero mean and variance $1$) with bounded
fourth moment, $1\leq k\leq N$,
\begin{equation}\label{eqn:fourth} 
\expect\left[\max_{1\leq k \leq N}\|\V{w}_k \|_\infty^4\right] < + \infty.
\end{equation}
These include standard
Rademacher (section~\ref{s_radenorm})
and sparse Rademacher vectors
(section~\ref{s_4rade}), as well
as standard Gaussian vectors (section~\ref{s_4gauss}).

\begin{theorem}\label{t_fourth}
Let $\M{A}\in \R^{n\times n}$ with $n\geq 3$
be symmetric, and 
\[\widehat{\M{A}}\equiv \frac1N\sum_{j=1}^N {\M{Aw}_k \V{w}_k^\top}\] 
be a Monte Carlo
estimator with independent random vectors
$\B{w}_k \in \R^n$, $1\leq k\leq N$, that have
independent elements 
with zero mean and variance 1.
If the vectors $\V{w}_k$ also 
have a bounded fourth moment (\ref{eqn:fourth}), then
\[ \expect\left[\|\D{ \widehat{\M{A}}} - \D{ \M{A}}\|_2^2\right]^{1/2}  \leq \|\M{A}\|_\infty\,\left(\sqrt{\frac{8e\ln n}{N}}  + \frac{8e\ln n}{N} \right) \left(\expect\left[\max_{1\leq k \leq N}  \|\V{w}_k\|_\infty^4\right]\right)^{1/2}.\]
\end{theorem}

\begin{proof}   
We make use of matrix concentration inequalities but follow the spirit of the analysis in~\cite{chen2012masked}.   

\paragraph{1. Symmetrization}
Write the normwise error by exploiting diagonal Schur products (\ref{e_diag})
\begin{align*}
\D{ \widehat{\M{A}}} - \D{\M{A}} = 
\M{I}\circ \widehat{\M{A}} - \M{I}\circ \M{A}
= \frac1N\sum_{k=1}^N{\left(\M{I}\circ (\M{A}\V{w}_k\V{w}_k^\top) - \M{I}\circ\M{A}\right)},
\end{align*}
and take expectations of the squared norms  
\begin{align*}
\expect\left[\left\|\D{ \widehat{\M{A}}} - \D{\M{A}}\right\|_2^2\right]=
\frac{1}{N^2}\,\expect\left[\left\|\sum_{k=1}^N{\left(\M{I}\circ (\M{A}\V{w}_k\V{w}_k^\top) - \M{I}\circ\M{A}\right)}\right\|_2^2\right].
\end{align*}
From the assumption that $\V{w}_k$ has
independent elements with zero mean and and variance~1 follows
$\expect\left[\M{I}\circ (\M{A}\V{w}_k\V{w}_k^\top)\right] = \M{I}\circ\M{A}$.  
Hence the matrix random variables 
\[\M{X}_k\equiv \M{I}\circ (\M{A}\V{w}_k\V{w}_k^\top)-\M{I}\circ\M{A},
\qquad 1\leq k\leq N\]
have zero mean.
We use symmetrization~\cite[Lemma 6.4.2]{vershynin2018high} to create \textit{symmetric} random variables $\varepsilon_k\, \M{X}_k$,
where $\varepsilon_k$ are independent symmetric Bernoulli random variables, 
that is, they are Rademacher variables 
as in Definition~\ref{d_rade}. The
Rademacher variables~$\varepsilon_k$ are independent of each other and also independent of the random vectors $\V{w}_k$. 
 Remark~\ref{r_rade} implies $\E[\varepsilon_k]=0$, hence 
 \begin{equation}\label{e_4zero}
\expect\left[\sum_{k=1}^N{\varepsilon_k\,\M{I}\circ{\M{A}}}\right]=\M{0}.
\end{equation}
Then \cite[Lemma 6.4.2]{vershynin2018high},
(\ref{e_4zero}) and (\ref{e_outer}) imply
\begin{align}\label{e_step1}
\expect\left[\left\| \sum_{k=1}^N \M{X}_k\right\|_2^2\right] 
\leq 2\, \expect\left[\left\| \sum_{k=1}^N{\varepsilon_k\,\M{I}\circ (\M{A}\V{w}_k\V{w}_k^\top)}\right\|_2^2\right]
= 2\, \expect\left[\left\| \sum_{k=1}^N{\M{Y}_k} \right\|_2^2\right],
\end{align}
where $\M{Y}_k\equiv \varepsilon_k\,\diag(\M{A}\V{w}_k)\diag(\V{w}_k)$
are symmetric random matrices, and
the second and third expectations range over all random vectors $\V{w}_k$ and
all Rademacher variables~$\varepsilon_k$.

\paragraph{2. Concentration inequality}
Applying Cauchy-Schwartz inequality and Theorem~\ref{thm:symmrand}
to the sum $\M{Z} \equiv\sum_{k=1}^N\M{Y}_k$ gives
\begin{align}\label{e_step2}
\expect[\|\M{Z}\|_2^2]^{1/2} \leq \sqrt{2e\ln{n}}\> \left\|\left(\sum_{i=1}^N{\expect[\M{Y}_k^2]}\right)^{1/2} \right\|_2 
+ 4e\ln{n}\> \expect\left[\max_{1\leq k \leq N}{\|\M{Y}_k\|_2^2}\right]^{1/2}. 
\end{align}
We bound the expectations that represent the matrix variance and 
the maximal two-norm separately.

\paragraph{3. Variance} 
As in item 3 of the proof of Theorem~\ref{thm:rademacher3} abbreviate 
\begin{align*}
\M{D}_{k}\equiv \diag(\M{Aw}_k), \qquad \M{W}_{k}\equiv \diag(\M{w}_k),\qquad
\M{O}\equiv \diag\left(\begin{bmatrix}\|\V{a}_1\|_1 & \cdots &\|\V{a}_n\|_1\end{bmatrix}\right).
\end{align*}
With this notation, $\expect[\M{Y}_k^2] = \expect [\M{D}_{k}^2 \M{W}_{k}^2]$. Consider the  diagonal term 
\[ (\M{D}_{k}^2 \M{W}_{k}^2)_{ii} = \left(\sum_{j=1}^n a_{ij} (\V{w}_k)_j\right)^2(\V{w}_k)_i^2 \leq \|\V{a}_i\|_1^2 \|\V{w}_k\|_\infty^4 \qquad 1 \leq i \leq n.\]
Therefore, $\M{D}_{k}^2 \M{W}_{k}^2 \preceq \|\V{w}_k\|_\infty^4 \M{O}^2 $. The symmetry of $\M{A}$
implies $\M{O}\preceq \|\M{A}\|_\infty\M{I}$.
Combining the two inequalities gives
\begin{align}\label{e_variance}
\expect[\M{Y}_k^2] = \expect [\M{D}_{k}^2 \M{W}_{k}^2] 
\preceq \|\M{A}\|_\infty^2 \>
\expect\left[\max_{1\leq k \leq N}\|\V{w}_k\|_\infty^4\right] \M{I}.
\end{align}
Now take square roots,
\begin{align*}
\left(\sum_{k=1}^N{\expect[\M{Y}_k^2]}\right)^{1/2} \preceq \sqrt{N}\> \|\M{A}\|_\infty\> 
\expect\left[\max_{1\leq k \leq N}{\|\V{w}_k\|_\infty^4}\right]^{1/2} \M{I},
\end{align*}
and bound the norm,
\begin{align}\label{e_step3}
\left\|\left(\sum_{k=1}^N{\expect[\M{Y}_k^2]}\right)^{1/2} \right\|_2 
\leq \sqrt{N}\>\|\M{A}\|_\infty\> 
\expect\left[\max_{1\leq k\leq N}{\|\V{w}_k\|_\infty^4}\right]^{1/2}.
\end{align}

\paragraph{4: Maximal two-norm} 
In analogy to (\ref{e_variance}), we derive
\begin{align*} 
\expect\left[\max_{1\leq k \leq N}{\|\M{Y}_k\|_2^2}\right] 
\leq \|\M{A}\|_\infty^2 \>
\expect\left[\max_{1\leq k \leq N}{\|\V{w}_k\|_\infty^4}\right], 
\end{align*}
and its  square root
\begin{align}\label{e_step4}
\expect\left[\max_{1\leq k \leq N}{\|\M{Y}_k\|_2^2}\right]^{1/2} 
\leq \|\M{A}\|_\infty\> \expect\left[\max_{1\leq k \leq N} {\|\V{w}_k\|_\infty^4}\right]^{1/2}. 
\end{align}

\paragraph{5. Putting everything together}  
Substitute the variance bound (\ref{e_step3}) and the norm bound (\ref{e_step4})
into the expectation (\ref{e_step2}) for the sum
\[\expect[\|\M{Z}\|_2]^{1/2} \leq 
\left(\sqrt{2e\ln{n}}\, N^{1/2}  + 4e\ln{n}\right)    \|\M{A}\|_\infty 
\>\expect\left[\max_{1\leq k \leq N}{\|\V{w}_k\|_\infty^4}\right]^{1/2}, \]
substitute this, in turn, into the expectation (\ref{e_step1}) for
the absolute error,
\begin{align*}
&\expect\left[\|\D{\widehat{\M{A}}} - \D{\M{A}}\|_2^2\right]^{1/2} 
\leq \frac{2}{N}\> \expect\left[\|\M{Z}\|_2\right]^{1/2} \\
&\qquad\qquad\leq\frac{2}{N}\>
\left(\sqrt{2e\ln{n}}\, N^{1/2}  + 4e\ln{n}\right)    \|\M{A}\|_\infty 
\>\expect\left[\max_{1\leq k \leq N}{\|\V{w}_k\|_\infty^4}\right]^{1/2},
\end{align*}
and simplify.
\end{proof}

\subsection{Gaussian vectors}\label{s_4gauss} We determine the minimal sampling amount for
Gaussian Monte Carlo estimators to be
normwise $(\epsilon,\delta)$ diagonal estimators.

\begin{corollary}\label{c_fourth}
Let $\M{A}\in \R^{n\times n}$ with $n\geq 3$
be symmetric, and let
\[\widehat{\M{A}}\equiv \frac1N\sum_{j=1}^N {\M{Aw}_k \V{w}_k^\top}\] be a Monte Carlo
estimator with independent Gaussian
random vectors $\V{w}_k\sim\mathcal{N}(\V{0}, \M{I})$ in $\real^n$, $1\leq k\leq N$.
Pick $\epsilon>0$.
For any $0<\delta<1$, if the sampling amount 
$N$ satisfies ${8e\ln n} \leq N \leq n$, and 
is at least 
\begin{align}\label{eqn:samp_gauss}
N\geq\frac{128\, (e\,\ln{n})^3}{\epsilon^2\,\delta}\>
\left(\frac{\|\M{A}\|_\infty}{\|\D{\M{A}}\|_{\infty}}\right)^2,
\end{align}
then $\|\D{\widehat{\M{A}}} -\D{ \M{A}}\|_2\leq 
\epsilon \| \D{ \M{A}}\|_2$
holds with probability at least $1-\delta$.
\end{corollary}

\begin{proof}
For Gaussian random vectors $\V{w}_k\sim\mathcal{N}(\V{0},\M{I})$,
\cite[(3.7)]{chen2012masked} implies 
\[  \left(\expect \left[\max_{1\leq k \leq N}{\|\V{w}_k\|_\infty^4} \right]\right)^{1/2} 
\leq  e \ln{(nN)}
\max_{1\leq i,j\leq n}{|\M{I}_{ij}|}=
e\ln{(nN)}. \]
Substituting this into Theorem~\ref{t_fourth} gives 
\[ \left( \expect \left[ \| \D{ \widehat{\M{A}}} - \D{ \M{A}}\|_2^2\right]\right)^{1/2}  \leq \left(\sqrt{\frac{8e\ln{n}}{N}}  + \frac{8e\ln{n}}{N} \right)\>
e\ln{(nN})\>\|\M{A}\|_\infty.  \]
Square both sides and apply Markov's inequality (Theorem~\ref{t_Markov}) to the random variable
$Z \equiv \| \D{ \widehat{\M{A}}} - \D{ \M{A}}\|_2$ using 
$t \equiv \epsilon \| \D{ \M{A}}\|_2$,
\begin{align}\label{e_fourth2}
&\prob\left[ \| \D{ \widehat{\M{A}}} - \D{ \M{A}}\|_2 
\geq \epsilon \| \D{ \M{A}}\|_2 \right] \\
\nonumber& \qquad\qquad\leq \left(\sqrt{\frac{8e\ln{n}}{N}}  + \frac{8e\ln{n}}{N} \right)^2\>
(e \ln{(nN)})^2\>
\frac{\|\M{A}\|_\infty^2}{\epsilon^2\, \|\D{\M{A}}\|_{\infty}^2}.  
\end{align}
Substituting the assumption ${8e\ln{n}} \leq N \leq n$ 
into the relevant part of the above bound gives 
\begin{align*}
\left(\sqrt{\frac{8e\ln{n}}{N}}  + \frac{8e\ln{n}}{N} \right)^2\>
(e \ln{(nN)})^2
\leq\left(2 \sqrt{\frac{8e\ln{n}}{N}}\right)^2 (2e\ln{n})^2
=\frac{128\, (e\,\ln{n})^3}{N}.
\end{align*}
Substitute this, in turn, into (\ref{e_fourth2}), set
the failure probability equal to 
\begin{align*}
\delta\equiv \frac{128\, (e\,\ln{n})^3}{\epsilon^2\,N}\>
\left(\frac{\|\M{A}\|_\infty}{\|\D{\M{A}}\|_{\infty}}\right)^2 
\end{align*}
and solve for the sampling amount~$N$.
\end{proof}

\section{Componentwise bounds}\label{s_comp}
We present componentwise bounds for Monte Carlo estimators based on standard 
Rademacher vectors (section~\ref{s_radecomp}), as
well as on standard
(section~\ref{s_5gauss})
and normalized Gaussian vectors
(section~\ref{s_5normgauss}).

Our attempt at deriving alternative normwise bounds by applying a union bound 
over the 
componentwise bounds for all diagonal elements did not produce results that were substantially tighter than our previous normwise bounds.

\subsection{Standard Rademacher 
vectors}\label{s_radecomp}
We present a componentwise worst case absolute error
bound (Corollary~\ref{c:rademacher1}), and a
bound on the minimal sampling amount that makes the Rademacher Monte Carlo
estimator a \textit{componentwise} ($\epsilon, \delta)$ diagonal estimator (Theorem~\ref{thm:rad:component}).

Theorem~\ref{thm:rademacher} 
is identical to the following worst case componentwise  bound.

\begin{corollary}\label{c:rademacher1}
Let $\M{A}\in \R^{n\times n}$ be non-diagonal symmetric, and let
\[K_1\equiv \|\D{\M{A}^2}-\D{\M{A}}^2\|_2,
\quad K_2 \equiv \|\M{A}-\D{\M{A}}\|_\infty, \quad  
d\equiv \frac{1}{K_1}\|\M{A}-\D{\M{A}}\|_F^2. \]
If $\widehat{\M{A}} \equiv\frac1N\sum_{k=1}^N \M{Aw}_k \V{w}_k^\top$ is a Monte Carlo estimator
with independent Rademacher vectors $\B{w}_k \in \R^n$, $1\leq k\leq N$, then
the probability that
the absolute error exceeds $t>0$ is at most
\[ \myP\left[\max_{1\leq i\leq n}{|a_{ii}-\widehat{a}_{ii}|} 
\geq t \right] \leq 8d\, 
\exp\left(\frac{-Nt^2}{2(K_1 + tK_2/3)}\right). \]
\end{corollary}

\begin{proof}
In Theorem~\ref{thm:rademacher}, the $p$-norm of the diagonal matrix 
$\D{\M{A}} -\D{\widehat{\M{A}}}$ is a largest
magnitude diagonal element.
\end{proof}

In contrast to Corollary~\ref{c:rademacher1}, the 
next bound depends on the particular diagonal element. We first require a lemma on the independence of products of Rademacher variables. 

\begin{lemma}\label{l5_rad}
    Let $Z_1,W_1,W_2,\ldots,W_n$ be independent Rademacher variables. Then the products 
    $X_1 \equiv ZW_1,\ldots, X_n \equiv ZW_n$
    are also independent Rademacher variables.
\end{lemma}
\begin{proof}
 For any $x_1,\ldots,x_n \in \{-1,+1\}$, using the law of total probability, the joint probability mass function satisfies 
 \begin{align*}
     \mathbb{P}\left[\cap_{i=1}^n\,\{X_i=x_i\}\right] &= \sum_{z\in \{-1,+1\}}
     \mathbb{P}\left[\cap_{i=1}^n\, \{X_i=x_i\} | Z = z\right]\mathbb{P}[Z = z]\\
     &= \frac{1}{2}\mathbb{P}\left[\cap_{i=1}^n\,\{W_i = -x_i\}\right]
     + \frac{1}{2}\mathbb{P}\left[\cap_{i=1}^n\,\{W_i = x_i\}\right]\\
     &= 2^{-n} = \prod_{i=1}^n\mathbb{P}\left[X_i=x_i\right].  
 \end{align*}
 Since the joint PMF factorizes, the variables $X_1,\ldots,X_n$ are independent. 
\end{proof}

\begin{theorem}\label{thm:rad:component}
Let $\M{A}\in \R^{n\times n}$ be non-diagonal  symmetric,
and let
\[\widehat{\M{A}} \equiv\frac1N\sum_{k=1}^N \M{Aw}_k \V{w}_k^\top\] be a Monte Carlo estimator
with independent Rademacher vectors $\B{w}_k \in \R^n$, $1\leq k\leq N$.
The probability that the absolute error
exceeds $t>0$ is at most
 \begin{equation}
\prob\left[\left|\widehat{a}_{ii} - a_{ii}\right| \geq t\right] \leq 2\exp\left(\frac{-N t^2}{2\,(\|\V{a}_i\|_2^2-a_{ii}^2)}\right), \qquad 1\leq j\leq n.
 \end{equation}
\end{theorem}

\begin{proof}
Fix $i$, for some $1\leq i\leq n$.
The properties in Remark~\ref{r_rade} allow us
to split off the original diagonal element from the estimator,
\begin{align*}
\widehat{a}_{ii} &= \frac{1}{N}\sum_{k=1}^N\left(\M{A}\M{w}_k\M{w}_k^{\top}\right)_{ii} = 
 \frac{1}{N}\sum_{k=1}^N\sum_{j=1}^n{a_{ij}\,
 (\V{w}_k)_i\,(\V{w}_k)_j} \\
 &= a_{ii} + \sum_{k=1}^N\sum_{j\neq i}\underbrace{\frac{a_{ij}}{N}(\V{w}_k)_i\,(\V{w}_k)_j}_{Z_{kj}^{(i)}}
 .
 \end{align*}
Lemma \ref{l5_rad} implies that for fixed $i$, the $Z_{kj}^{(i)}$ are independent.
 Remark~\ref{r_rade} implies that they
 have zero mean, and are bounded by
 \begin{align*}
-\tfrac{|a_{ij}|}{N} \leq Z_{kj}^{(i)}
\leq \tfrac{|a_{ij}|}{N}, \qquad 1\leq k\leq N,\quad 1\leq j \leq n,\quad j\neq i. 
\end{align*}
 Hence the absolute error
$\widehat{a}_{ii}-a_{ii}$
is a sum of independent bounded zero-mean random variables, and we can apply
 Hoeffding's inequality in Theorem~\ref{t_Hoeffding}
 	\begin{equation*}
\prob\left[\left|\widehat{a}_{ii} - a_{ii}\right| \geq t\right] 
\leq 2\exp\left(\frac{-2t^2}{\sum_{k=1}^N\sum_{j\neq i}
\left(\frac{2}{N}{|a_{ij}|}\right)^2}\right)	= 2\exp\left(\frac{-Nt^2}{2\sum_{j\neq i}a_{ij}^2}\right).
	\end{equation*}
At last, write 
$\sum_{j\neq i}a_{ij}^2=
\|\V{a}_i\|_2^2-a_{ii}^2$.
 \end{proof}

Theorem~\ref{thm:rad:component} implies that the accuracy for estimating
a single diagonal element depends only on the magnitude of the off-diagonal elements in the corresponding row and column. 

 We determine the minimal sampling amount
required to make 
the Rademacher Monte Carlo estimator a componentwise ($\epsilon,\delta)$ 
diagonal estimator. {For symmetric matrices, this result coincides with the bound in Equation 40 of \cite{BN22} but uses a different proof technique.}

\begin{corollary}\label{c:radcomp}
Let $\M{A}\in \R^{n\times n}$ be non-diagonal symmetric,
and let
\[\widehat{\M{A}} \equiv\frac1N\sum_{k=1}^N \M{Aw}_k \V{w}_k^\top\] be a Monte Carlo estimator
with independent Rademacher vectors $\B{w}_k \in \R^n$, $1\leq k\leq N$.
  Pick $\epsilon > 0$, and a diagonal element $a_{ii}\neq 0$ of $\M{A}$. For any $0< \delta < 1$, if the sampling amount is at least 
 \[N\geq 
 \left(\frac{\|\V{a}_i\|_2^2-a_{ii}^2}{a_{ii}^2}\right)
 \frac{2\ln(2/\delta)}{\epsilon^2},\]
 then $|a_{jj}-\widehat{a}_{jj}|\leq \epsilon |a_{jj}|$ holds with probability at least $1-\delta$. 
\end{corollary}

\begin{proof}
Define the 2-norm offdiagonal column sums 
\begin{equation*}
\off_i\equiv (\|\V{a}_i\|_2^2 -a_{ii}^2)^{1/2},\qquad 1\leq i\leq n,
\end{equation*}
and denote the bound for the failure probability in Theorem \ref{thm:rad:component} by 
\[\delta \equiv 2\exp\left(\frac{-Nt^2}{2\,\off_i^2}\right), 
\]
and solve it for $t$, 
\[t = \sqrt{\frac{2\,\off_i^2}{N}\ln(2/\delta)}.\]
Restate Theorem \ref{thm:rad:component} in terms of the failure probability: With probability at most $1-\delta$, the absolute error of a specific diagonal element is bounded above by 
\[|a_{ii}-\widehat{a}_{ii}|\leq t = \sqrt{\frac{2\,\off_i^2}{N}\ln(2/\delta)}. \]
Converting this absolute error into a relative error requires $t\leq \epsilon |a_{ii}|$, which implies
\[N\geq 
\left(\frac{\off_i}{|a_{ii}|}\right)^2
\frac{2\ln(2/\delta)}{\epsilon^2}.\]
\end{proof}

The minimal sampling amount for
computing $a_{ii}$ with
the Rademacher Monte Carlo estimator
depends on
$\frac{\|\V{a}_i\|_2^2-a_{ii}^2}{|a_{ii}|}$, which
represents the relative 2-norm deviation of the $i$th column
and row of $\M{A}$ from diagonality. 
Thus, the more diagonal the $i$th row and column, the fewer samples are required for 
a $(\epsilon, \delta)$ estimator.

\subsection{Gaussian vectors}\label{s_5gauss}
We present a componentwise absolute error
bound (Theorem~\ref{thm:gauss:component})
for Gaussian Monte Carlo estimators,
and a bound on the minimal sampling amount that makes the Gaussian 
Monte Carlo estimator a componentwise $(\epsilon,\delta)$ estimators
(Corollary~\ref{c:gauss:component}).
Our bounds are derived from and identical to bounds for trace estimators in  \cite{cortinovis2021randomized}.

\begin{theorem}\label{thm:gauss:component}
Let $\M{A}\in\real^{n\times n}$ be non-diagonal symmetric,
\begin{equation*}
L_{1i}\equiv |a_{ii}|+\|\V{a}_i\|_2,\qquad
L_{i2}\equiv |a_{ii}|^2+\|\V{a}_i\|_2^2,
\qquad 1\leq i\leq n,
\end{equation*}
and let
$\widehat{\M{A}}\equiv \frac{1}{N}\sum_{k=1}^N\M{A}\M{z}_k\M{z}_k^\top \in \R^{n\times n}$ 
be a Monte Carlo estimator with 
independent Gaussian vectors
$\M{z}_k\sim\mathcal{N}(\V{0},\M{I})$ in $\R^n$, $1\leq k\leq N$. 
If $t > 0$, then
\begin{equation}\label{eqn:gauss_bound}
\prob\left[\left|\widehat{a}_{ii} - a_{ii}\right| > t\right] \leq 		2\exp\left(\frac{-Nt^2}{2(L_{i2} + t\,L_{i1})}\right), \quad 1\leq i\leq n.
	\end{equation}
\end{theorem}

\begin{proof}
Fix $i$ for some $1\leq i\leq n$.
Write the diagonal element as an inner product
 \[(\M{A}\M{z}_k\M{z}_k^\top)_{ii} = \sum_{j=1}^n{a_{ij}(\M{z}_k)_i(\M{z}_k)_j} = 
 \M{z}_k^\top \M{B}_i\M{z}_k,\qquad 1\leq k\leq N\]
involving the symmetric matrix
  	\[\M{B}_i \equiv \begin{bmatrix}
		0& & \frac{1}{2}a_{1i} & & 0\\
		& & \vdots && \\
		\frac{1}{2}a_{i1} &\cdots & a_{ii} & \cdots & \frac{1}{2}a_{in}\\
		& & \vdots && \\
		0& & \frac{1}{2}a_{ni} & & 0\\
\end{bmatrix}\in\rnn \quad \text{with}\quad \trace(\M{B}_i)=a_{ii}.\]
We can think of
\begin{equation*}
\widehat{a}_{ii}=
\left(\frac{1}{N}\sum_{k=1}^N{\M{A}\V{z}_k\V{z}_k^\top}\right)_{ii}=
\frac{1}{N}\sum_{k=1}^N{\left(\M{A}\V{z}_k\V{z}_k^\top\right)_{ii}}=
\frac{1}{N}\sum_{k=1}^N
{\M{z}_k^\top \M{B}_i\M{z}_k}
\end{equation*}
as a Monte Carlo estimator
for $\trace(\M{B}_i)$, and apply
the bound for Gaussian trace estimators
\cite[Theorem~1]{cortinovis2021randomized} 
	\begin{equation*}
	\prob\left[\left|\widehat{a}_{ii} - a_{ii}\right| \geq t\right] \leq 		2\exp\left(\frac{-Nt^2}{4\|\M{B}_i\|_F^2 + 4t\|\M{B}_i\|_2}\right),
	\end{equation*}
where	 
	\[ \|\M{B}_i\|_F^2 = \frac{1}{2}\left(|a_{ii}|^2 + \|\M{a}_i\|_2^2\right)=\frac{L_{i2}}{2},\qquad \|\M{B}_i\|_2 = \frac{1}{2}\left(|a_{ii}| + \|\M{a}_i\|_2 \right)=\frac{L_{i1}}{2}.\]
\end{proof}

Theorem~\ref{thm:gauss:component}
 implies that with Gaussian vectors, the accuracy for estimating
a single diagonal element depends on the magnitude of  all elements in the corresponding row and column. By contrast, the bounds of Theorem~\ref{thm:rad:component} for Rademacher vectors depend only on the magnitude of the off-diagonal elements. 

We determine the minimal sampling amount required to make the 
Gaussian Monte Carlo estimator a componentwise $(\epsilon,\delta)$ estimator.

\begin{corollary}\label{c:gauss:component}
Let $\M{A}\in\real^{n\times n}$ be non-diagonal symmetric,
\begin{equation*}
\Delta_{1i}\equiv 1+\frac{\|\V{a}_i\|_2}{|a_{ii}|},\qquad
\Delta_{i2}\equiv 1+\left(\frac{\|\V{a}_i\|_2}{|a_{ii}|}\right)^2,
\qquad 1\leq i\leq n,
\end{equation*}
and let
$\widehat{\M{A}}\equiv \frac{1}{N}\sum_{k=1}^N\M{A}\M{z}_k\M{z}_k^\top \in \R^{n\times n}$
be a Monte Carlo estimator with 
independent Gaussian vectors
$\M{z}_k\sim\mathcal{N}(\V{0},\M{I})$ in $\R^n$, $1\leq k\leq N$. 
Pick $\epsilon>0$, and a diagonal element $a_{ii}\neq 0$ of $\M{A}$. For any $0<\delta<1$, if the sampling amount is at
least
\begin{equation*}
N\geq (\Delta_{2i}+\Delta_{1i}\epsilon)\,\frac{2\ln{(2/\delta)}}{\epsilon^2}
\end{equation*}
then $|\widehat{a}_{ii} - a_{ii}| \leq \epsilon\,|a_{ii}|$ holds with probability
at least $1-\delta$.
\end{corollary}

\begin{proof}
This follows immediately from the lower bound for $N$ in \cite[Theorem~1]{cortinovis2021randomized}. 
\end{proof}

The required sampling amount for computing $a_{ii}$ with the Gaussian Monte Carlo
estimator depends on $\|\V{a}_i\|_2/|a_{ii}|$ which can be interpreted as the 2-norm
derivation of the $i$th column and row of $\M{A}$ from diagonality. The more diagonal
the $i$th row and column, the smaller the sampling amount for the $(\epsilon,\delta)$ estimator.

\subsection{Normalized Gaussian vectors}\label{s_5normgauss}
We extend and complete the analysis 
in~\cite{bekas2007estimator} for a Monte Carlo estimator based on normalized
Gaussian vectors,
\begin{equation}\label{e_5MC}
\widehat{\M{A}} \equiv \left(\sum_{k=1}^N\M{A}\M{z}_k\M{z}_k^\top\right) \oslash \left(\sum_{k=1}^N\M{z}_k\M{z}_k^\top\right),
\end{equation}
where $\M{z}_k \in \R^n$ are independent random
vectors, and $\oslash$ denotes elementwise division. 
We derive the 
distribution of the componentwise absolute errors  (Lemma~\ref{l_5bound5}),
followed by a bound (Theorem~\ref{t_5bound6}).

We represent the distribution for the absolute errors in the
diagonal elements in terms
of a \textit{Student $t$-distribution} with $N\geq 1$ degrees of freedom \cite[Definition 7.3.3]{Larsen2006},
\begin{equation} \label{eqn:studentt}
T_N\equiv \frac{Z}{\sqrt{U/N}},
\end{equation}
where $Z$ is a Gaussian $\mathcal{N}(0,1)$
random variable, and $U$ an independent chi-square random variable
with $N$ degrees of freedom.

\begin{lemma}\label{l_5bound5}
Let $\M{A}\in\rnn$ be symmetric, and (\ref{e_5MC})
be a Monte Carlo estimator
where $\M{z}_k\sim\mathcal{N}(\V{0},\M{I})$
in $\R^n$ are independent Gaussian random
vectors, $1\leq k\leq N$.
The absolute errors in the diagonal elements
are distributed as 
    \begin{equation*}\label{eqn:tdist_err}
\widehat{a}_{ii} - a_{ii}\  \sim\  \sqrt{\frac{\|\V{a}_i\|_2^2 - a_{ii}^2}{N}}\, T_N, 
			\qquad 1\leq i\leq n.
		\end{equation*}
\end{lemma}

\begin{proof}
Due to the normalization in the 
denominator, we can extract the diagonal elements of $\M{A}$ from the diagonal elements 
of the Monte Carlo estimator $\widehat{\M{A}}$,
\begin{equation*}\label{eqn:bekas_estimator}
\widehat{a}_{ii} = \frac{\sum_{k=1}^N{\sum_{j=1}^n{a_{ij}(\M{z}_k)_i
(\M{z}_k)_j}}}{\sum_{k=1}^N{(\M{z}_k)_i^2}} = a_{ii} + 
\frac{\sum_{k=1}^N{\sum_{j\neq i}
{a_{ij}\,(\M{z}_k)_i(\M{z}_k)_j}}}{\sum_{k=1}^N{(\M{z}_k)_i^2}}, \qquad 1\leq i\leq n.
\end{equation*}
Normalize across the $i$th elements of the Gaussian vectors $\V{z}_k$ to unit vectors $\M{u}_i \in \mathbb{R}^N$ 
with elements
\[(\M{u}_i)_k \equiv \frac{(\M{z}_k)_i}{\sqrt{\sum_{k'=1}^N{(\M{z}_{k'})_i^2}}},
\qquad 1\leq k\leq N, \quad 1\leq i\leq n.
\]
Use the denominator to normalize the $i$th component in the $i$th 
absolute error,
\begin{equation}\label{e_nG1}
\widehat{a}_{ii}-a_{ii} = \frac{\sum_{k=1}^N\sum_{j\neq i}
{a_{ij}\,(\M{z}_k)_j(\M{u}_i)_k}}
{\sqrt{\sum_{k'=1}^N{(\M{z}_{k'})_i^2}}},
\qquad 1\leq i\leq n.
\end{equation}

The rotational invariance of the  standard Gaussian distribution guarantees the
independence of the direction vectors
$\M{u}_i$ and radial components $\sqrt{\sum_{k=1}^N(\M{z}_k)_i^2}$; see~\cite[Exercise  3.3.6]{vershynin2018high}.
Hence, the numerator and denominator in (\ref{e_nG1})
are independent. We can rewrite~\eqref{e_nG1} as 
\[ \widehat{a}_{ii}-a_{ii} = \frac{Z_i}{\sqrt{U_i/N}},\]
where we define the random variables $Z_i \equiv\sum_{k=1}^N\sum_{j\neq i}
{a_{ij}\,(\M{z}_k)_j(\M{u}_i)_k}/\sqrt{N}$ and $U_i \equiv \sum_{k=1}^N(\M{z}_k)_i^2$ for $1 \leq i \leq n$. 
The random variable $U_i$ has a chi-square distribution with $N$ degrees of freedom. 
The conditional distribution of $Z_i$ given $\M{u}_i$ is Gaussian (see e.g.~Exercise 3.3.3(a) of \cite{vershynin2018high}), with zero mean and variance 
\[\frac1N\sum_{k=1}^N \sum_{j\neq i}a_{ij}^2(\B{u}_i)_k^2 = \frac1N\left(\sum_{j\neq i}a_{ij}^2\right)\left(\sum_{k=1}^N(\B{u}_i)_k^2\right) = \frac1N\sum_{j\neq i}a_{ij}^2 = \frac1N( \|\M{a}_i\|_2^2-a_{ii}^2).\]
Therefore, $Z_i | \B{u}_i \sim \mathcal{N}(0,\frac1N (\|\M{a}_i\|_2^2-a_{ii}^2))$. However, the conditional distribution is independent of $\B{u}_i$, so this is also the unconditional distribution. The claim then follows from~\eqref{eqn:studentt}.
\end{proof}

\begin{remark}
{
It was observed in \cite{BN22} that the square of the error $(\widehat{a}_{ii}-a_{ii})^2$ has a scaled $F$-distribution. We note that the square of a Student $t$-distribution is specifically a scaled $F$-distribution with one degree of freedom in the numerator. Moreover, for a single sample $N=1$ the error has a Cauchy distribution, which has undefined mean and variance.}
\end{remark}

If $N$ is large, then the $t$-distribution $T_N$ can be approximated by a standard normal distribution. However, $T_N$ has wider tails, thus somewhat weaker tail bounds. Existing tail bounds for the Student $t$-distribution 
imply the following concentration inequality
for  error bounds.

\begin{theorem}\label{t_5bound6}
Let $\M{A}\in\rnn$ be symmetric, and (\ref{e_5MC})
be a Monte Carlo estimator
where $\M{z}_k\sim\mathcal{N}(\V{0},\M{I})$
in $\R^n$ are independent Gaussian random
vectors, $1\leq k\leq N$.
For any $t > 0$, 
 \begin{equation*} \label{eqn:bekas_bound}
     \prob\left[|\widehat{a}_{ii} - a_{ii}| > t\right] \leq\sqrt{\frac{2\,(\|\V{a}_i\|_2^2 - a_{ii}^2)}{\pi N}}\,
     \frac{1}{t}\,
     \left(1 + \frac{t^2}{\|\V{a}_i\|_2^2 - a_{ii}^2}\right)^{-\frac{N-1}{2}},\qquad 1\leq i\leq n.
 \end{equation*}
\end{theorem}

\begin{proof}
The probability density function of $T_N$ is \cite{soms1976asymptotic}
\[f_N(x)=c_N\left(1+\frac{x^2}{N}
\right)^{-\frac{N+1}{2}}
	\qquad \text{where}\quad c_N \equiv \frac{\Gamma((N+1)/2)}{\Gamma(N/2)\sqrt{N\pi}},\]	
	and $1/\pi \leq c_N\leq 1/\sqrt{2\pi}$. \
If $F_N(x)$ is the cumulative distribution 
function for $T_N$ then by
	\cite[Theorem~3.1]{soms1976asymptotic} 
	\begin{equation*}
	\begin{aligned}
\prob\left[(\widehat{a}_{ii} - a_{ii}) > \sqrt{\frac{\|\V{a}_i\|_2^2 - a_{ii}^2}{N}} x  \right] = & \> 1-F_N(x)<\frac{f_N(x)}{x}
\left(1+\frac{x^2}{N}\right) \\
 = & \> \frac{c_N}{x}\left(1+\frac{x^2}{N}
 \right)^{-\frac{N-1}{2}}.
 \end{aligned}
	\end{equation*}
Take $x = t\,\sqrt{\frac{N}{\|\V{a}_i\|_2^2 - a_{ii}^2}}$ and bound the upper tail with $c_N\leq \sqrt{\frac{1}{2\pi}}$, to obtain
\[\prob\left[ (\widehat{a}_{ii} - a_{ii}) > t   \right] \leq \frac{1}{ t}\sqrt{\frac{\|\V{a}_i\|_2^2 - a_{ii}^2}{2\pi N}} \left(1 + \frac{t^2}{\|\V{a}_i\|_2^2 - a_{ii}^2}\right)^{-\frac{N-1}{2}}.\] 
Since $T_N$ is symmetric about the origin, the lower tail has the same bound. Using a union bound  gives the desired inequality.
\end{proof}

We determine a sampling amount sufficient to make the normalized Gaussian Monte Carlo estimator a componentwise $(\epsilon,\delta)$ estimator.

\begin{corollary}
Let $\M{A}\in \mathbb{R}^{n\times n}$ be non-diagonal symmetric, and let 
\[\Psi_{i} \equiv \frac{|a_{ii}|}{(\|\V{a}_i\|_2^2-a_{ii}^2)^{1/2}}, \quad 1\leq i \leq n,\]
and let $\widehat{\M{A}}$ be defined as in \eqref{e_5MC}. Pick $\epsilon > 0$ and a diagonal element $a_{ii}\neq 0$ of $\M{A}$. For any $0<\delta < 1$, if the sampling amount is positive and at least 
\[N \geq 1 + 2\ln\left(\frac{\sqrt{2/\pi}}{\delta \epsilon \Psi_i}\right)/\ln(1+\epsilon^2\Psi_i^2),\]
then $|\widehat{a}_{ii}-a_{ii}|\leq \epsilon |a_{ii}|$ holds with probability at least $1-\delta$.
\end{corollary}
\begin{proof}
 In Theorem \ref{t_5bound6}, set $t = \epsilon |a_{ii}|$. If the sampling number satisfies the desired bound, it follows that 
 \begin{align*}
     \prob\left[ (\widehat{a}_{ii} - a_{ii}) > t   \right] &\leq \sqrt{\frac{2\,(\|\V{a}_i\|_2^2 - a_{ii}^2)}{\pi N}}\,
     \frac{1}{t}\,
     \left(1 + \frac{t^2}{\|\V{a}_i\|_2^2 - a_{ii}^2}\right)^{-\frac{N-1}{2}}\\
     &= 
     \frac{\sqrt{2/\pi}}{\epsilon \Psi_i\sqrt{ N}} \left(1 + \epsilon^2\Psi_i^2\right)^{-\frac{N-1}{2}}\\
     & \leq \frac{\sqrt{2/\pi}}{\epsilon \Psi_i} \left(1 + \epsilon^2\Psi_i^2\right)^{-\frac{N-1}{2}}
 \end{align*}
where the final inequality holds since $N\geq 1$ by assumption. Set the failure probability $\delta$ to the right hand side as 
\[ \delta \equiv \frac{\sqrt{2/\pi}}{\epsilon \Psi_i} \left(1 + \epsilon^2\Psi_i^2\right)^{-\frac{N-1}{2}},\]
and solve for $N$. 
\end{proof}
The larger the value of $\Psi_i$ (the same measure of diagonal dominance that appears in Corollary \ref{c:radcomp}), the smaller the sampling amount for the ($\epsilon$, $\delta$) estimator.

\section{Application: Monte Carlo estimators for a derivative-based global sensitivity metric}\label{s_appl}
We bound the absolute error (Theorem~\ref{t_66})
in a Monte Carlo estimator for global sensitivity analysis,
and more specifically for a derivative-based global sensitivity metric (DGSM)
of a function $f :\R^n\rightarrow \R$  whose partial derivatives are square integrable with respect to a probability density function
$\rho_{\matc{X}}(\V{x})$. The DGSM is equal to the diagonal $\D{\M{C}}$ 
of the matrix
\begin{align*}
\M{C} = \int_{\mathcal{X}} \nabla f(\V{x}) [\nabla f(\V{x})]^\top \rho_{\matc{X}} (\V{x}) d\V{x}.  
\end{align*}
The matrix $\M{C}$ is well-defined, symmetric positive semidefinite,
and can be interpreted as a second moment matrix of the gradient.
We compute the DGSM with the  Monte Carlo estimator
\begin{align}\label{e_6MC}
\widehat{\M{C}} \equiv \frac1N \sum_{k=1}^N \V{z}_k\V{z}_k^\top \qquad
 \text{where}\quad \V{z}_k \equiv\nabla f(\V{x}_k)
 \end{align}
 and $\V{x}_k$, $1\leq k\leq N$, are independent samples from the distribution 
of~$\rho_{\matc{X}}(\V{x})$. Below is a normwise bound
 for the error in the DGSM computed by the Monte Carlo estimator~(\ref{e_6MC}). Its
 derivation is related to the analysis in~\cite[Section 4]{lam2020multifidelity}.

\begin{theorem}\label{t_66}
Let $f :\R^n\rightarrow \R$ 
have square integrable partial derivatives with respect to the probability 
density function $\rho_{\matc{X}}(\V{x})$,
$\|\nabla f\|_\infty \leq \beta $ almost surely,
$\widehat{\M{C}}$ be the Monte Carlo estimator in (\ref{e_6MC}), and
\begin{align*}
c_{\max} &\equiv \|\D{\M{C}}\|_2
\qquad 
S_1 \equiv \|\D{\M{C}}
\left(\beta^2\M{I}-\D{\M{C}}\right)\|_2\\
S_2 &\equiv c_{\max}+\beta^2,\qquad 
d \equiv \frac{\sum_{i=1}^n c_{ii}( \beta^2  - c_{ii})}{S_1}. 
\end{align*}
If $c_{\max}> 0$ and $S_1 > 0$, then
\begin{align*}
\prob\left[\|\D{\M{C}} - \D{\widehat{\M{C}}}\|_2  \geq t \right] 
\leq 8d\exp \left(\frac{-t^2/2}{S_1 + S_2t/3}\right).
\end{align*}
\end{theorem}

\begin{proof}
Before applying the matrix Bernstein inequality in 
Theorem~\ref{thm:bernstein2}, we need to verify the assumptions.
The Monte Carlo estimate $\D{\widehat{\M{C}}}$
is an unbiased estimator of the DGSM $\D{\M{C}}$ whose largest diagonal
element is
$$c_{\max} = \|\D{\M{C}}\|_2 = \max_{1 \leq i \leq n}{|c_{ii}|} =  \max_{1 \leq i \leq n}\expect\left[(\nabla     f(\matc{X}))_i^2\right] \leq \beta^2.$$ 
The absolute error in the DGSM computed by the Monte Carlo estimator
(\ref{e_6MC}) is
\[ \M{Z} = \D{\M{C}} - \D{\widehat{\M{C}}} = \sum_{k=1}^N \M{S}_k,\qquad
\text{where}\quad
\M{S}_k \equiv \frac1N\left (\D{\M{C}} - \D{\V{z}_k\V{z}_k^\top}\right).\] 
The summands $\M{S}_k$ have zero mean and are bounded by 
\begin{align*}
\|\M{S}_k\|_2 &\leq  
\frac1N \left(\|\D{\M{C}}\|_2 + \|\D{\V{z}_k\V{z}_k^\top}\|_2\right)\\
&\leq \frac{\|\D{\M{C}}\|_2+\beta^2}{N}  = \frac{c_{\max}+\beta^2}{N}=\frac{S_2}{N},
\qquad 1\leq k\leq N.
\end{align*}
We let $L = S_2/N$ so that $\|\M{S}_k\|_2 \leq L$. The variance is 
\begin{equation}\label{e_6v}
 \Vy[\M{Z}] = \sum_{k=1}^N \expect [\M{S}_k^2] = \frac{1}{N^2} \sum_{k=1}^N \expect \left[\left(\D{\M{C}} - \D{\V{z}_k\V{z}_k^\top}\right)^2\right]. 
 \end{equation}
Linearity of the expectation, 
the majorization $\D{\V{z}_k\V{z}_k^\top} \preceq \beta^2  \M{I}$,
$\expect [\D{\V{z}_k\V{z}_k^\top}] = \D{\M{C}}$ 
and commutativity of diagonal matrices imply for the summands,
\begin{align*} 
\expect\left[\left(\D{\M{C}} - \D{\V{z}_k\V{z}_k^\top}\right)^2\right] 
&= \expect\left[ \D{\M{C}}^2 + (\D{\V{z}_k\V{z}_k^\top})^2 - 2\D{\M{C}}\,\D{\V{z}_k\V{z}_k^\top} \right] \\
&=\expect\left[ (\D{\V{z}_k\V{z}_k^\top})^2 \right]-\D{\M{C}}^2\\
&\preceq   \beta^2 \D{\M{C}} - \D{\M{C}}^2, \qquad 1\leq k\leq N.
\end{align*}
Substitute the above into \eqref{e_6v}
  \begin{align*}
 \Vy[\M{Z}] \preceq \M{V} \equiv\frac1N 
 \left(\beta^2 \D{\M{C}} - \D{\M{C}}^2\right),
  \end{align*}
and apply Theorem~\ref{thm:bernstein2} 
with $\nu = \|\M{V}\|_2 = S_1/N$ and $d = \intdim(\M{V})$.
\end{proof}

Below is the minimal sampling amount that
makes the Monte Carlo estimator in~(\ref{e_6MC})
a normwise $(\epsilon,\delta)$ diagonal estimator.

\begin{corollary}\label{c_66}
Let $f :\R^n\rightarrow \R$ 
have square integrable partial derivatives with respect to the probability 
density function $\rho_{\matc{X}}(\V{x})$,
$\|\nabla f\|_\infty \leq \beta$ almost surely,
$\widehat{\M{C}}$ be the Monte Carlo estimator in (\ref{e_6MC}), and
\begin{align*}
c_{\max} &\equiv \|\D{\M{C}}\|_2, \qquad 
S_1 \equiv \|\D{\M{C}}\left(\beta^2 \M{I}-\D{\M{C}}\right)\|_2\\
S_2 &\equiv c_{\max}+\beta^2,\qquad 
d \equiv \frac{\sum_{i=1}^n c_{ii}}{S_1}. 
\end{align*}
Pick $\epsilon > 0$.
 If $c_{\max}> 0$ and $S_1 > 0$, then
for any $0 < \delta < 1$, if the
sampling amount is at least
\[ N \geq \frac{S_2}{3\epsilon^2 } \left(2\epsilon  + \frac{6S_1}{ c_{\max}\,S_2}\right)  \ln(8d/\delta),\]
then $\|\D{\M{C}} - \D{\widehat{\M{C}}}\|_2 \leq  \epsilon\|\D{\M{C}}\|_2$ holds with probability at least $1-\delta$.
\end{corollary}

\begin{proof}
The proof is similar to that of \cref{c:rademacher} but is based instead on~\cref{t_66}. 
\end{proof}

\subsection{Illustrative examples}\label{s_6ex}
We determine the constants in Corollary~\ref{c_66}
for two different functions
$f :\mb{R}^n \rightarrow \mb{R}$,
and random variables
$\matc{X} \in \real^n$ 
from a uniform distribution over $\mathcal{X} = [-1,1]^n$.

\paragraph{Linear Function} Let $f(\V{x}) = \V{h}^\top\V{x}$ with $\V{h}\in\real^n$.
Then $\|\nabla f\|_\infty \leq \beta\equiv \|\V{h}\|_\infty$ almost surely. The second moment matrix $\M{C} = \V{hh}^\top$ has a largest diagonal entry $c_{\max} = \|\V{h}\|_\infty^2$. Let $\V{v} \in \R^n$ have entries $v_i =  h_i^2 \left(\|\V{h}\|_\infty^2 -  h_i^2\right)$ for $1 \leq i \leq n$. The  constants in Corollary~\ref{c_66} are
\begin{equation*}
S_1  =  \|\V{v}\|_\infty, \quad S_2 = 2\|\V{h}\|_\infty^2, \quad d = \frac{\sum_{i=1}^n v_i}{S_1}. 
\end{equation*}

\paragraph{Quadratic function} Let $f(\V{x}) = \frac12 \V{x}^\top\M{S}\V{x}$ where $\M{S}\in\real^{n\times n}$ 
$\M{S} = \M{S}^\top$
is a symmetric square root of the positive semidefinite matrix $\M{M}=\M{S}^2$. Then $\|\nabla f\|_\infty \leq 
\beta\equiv\|\M{S}\|_\infty$ almost surely. The second moment matrix $\M{C} = \frac13 \M{M}$ has a largest diagonal element $c_{\max} = \frac13\|\D{\M{M}}\|_{\infty}$. Let $\V{v} \in \R^n$ be have entries $v_i =  \frac{1}{3}m_{ii}\left( \|\M{S}\|_\infty^2 - \frac13 m_{ii}\right)$ for $1 \leq i \leq n$. The constants in Corollary~\ref{c_66} are
\begin{equation*}
S_1  =   \|\V{v}\|_\infty, 
\quad S_2  = \frac{1}{3}\,\|\D{\M{M}}\|_\infty + \|\M{S}\|_\infty^2, \quad d =   
\frac{\sum_{i=1}^n v_{i}}{S_1}.
\end{equation*}

\section{Numerical Experiments}\label{s_num}
After describing our test matrices (section~\ref{s_test}), we present 
four different types of
numerical experiments to illustrate
the accuracy of the Monte Carlo estimators:
Rademacher Monte Carlo estimators applied to 
the test matrices (section~\ref{s_exp1}), accuracy 
of different Monte Carlo estimators (section~\ref{e_exp2}), 
effect of the sparsity on the accuracy of Rademacher Monte Carlo  
estimators (section~\ref{s_exp3}), and
accuracy of the DGSM Monte Carlo estimator (section~\ref{s_exp4}).

\subsection{Test Matrices}\label{s_test}
We perform numerical experiments
on three symmetric test matrices 
from~\cite{roosta2015improved}
of dimension $n = 100$ that depend on a 
parameter $\theta$.  

\begin{enumerate}
    \item Identity plus rank-1
\begin{align*}
\M{A} = \M{I} + \theta \V{e}\V{e}^\top\qquad \text{where} \quad
.01\leq \theta \leq 0.1,
\end{align*}
where $\V{e}\in \R^n$ is a vector of ones. The constants in Corollary~\ref{c:rademacher} are
\[ K_1 = (n-1) \theta^2, \qquad K_2 = (n-1)\theta, \qquad \|\D{\M{A}}\|_\infty =  1+\theta,  \]
so that 
\[ \Delta_1 = \frac{K_1}{(1+\theta)^2}, \qquad \Delta_2 = \frac{(n-1)\theta}{1+\theta}, \qquad d = n. \]

\item Rank-1 with decaying elements 
    \[ \M{A} = \frac{\V{xx}^\top}{\|\V{x}\|_2^2} \qquad 
\text{where}\quad
    \V{x}_j = e^{-j(1-\theta)}, \quad 1 \leq j \leq n,
  \qquad 0.1\leq \theta\leq 1.\]
The constants in Corollary~\ref{c:rademacher} are
\begin{align*}
K_1 &= \left(\frac{x_1}{\|\V{x}\|_2}\right)^2\left( 1 - \left(\frac{x_1}{\|\V{x}\|_2}\right)^2 \right),\qquad
K_2 = \frac{x_1}{\|\V{x}\|_2^2} \sum_{j > 1} x_j
\end{align*}
and $\|\D{\M{A}}\|_\infty = \left(\frac{x_1}{\|\V{x}\|_2}\right)^2$so that
  \[\Delta_1 = \left(\frac{\|\V{x}\|_2}{x_1}\right)^2-1, \qquad \Delta_2 =\frac{ \sum_{j > 1} x_j}{x_1}, \qquad d = \frac{\sum_{i=1}^n x_i^2 \left(\|\V{x}\|_2^2 - x_i^2 \right)}{x_1^2 ( \|\V{x}\|_2^2 - x_1^2)}.  \]
  
\item Tridiagonal Toeplitz matrix 
   \[ \M{A} = \begin{bmatrix} 1 & \theta \\ \theta & 1 & \ddots  \\ &  \ddots & \ddots & \theta \\ & & \theta & 1\end{bmatrix}
\qquad\text{where}\quad 0.1\leq \theta\leq 1.\]
The constants in Corollary~\ref{c:rademacher} are
\begin{align*}
K_1 = 2\theta^2 , \qquad 
K_2 = 2\theta, \qquad 
\|\D{\M{A}}\|_\infty = 1
\end{align*}
so that
\[\Delta_1 = 2\theta^2 , \qquad \Delta_2 = 2\theta, \qquad d = \frac{2(n-1) \theta^2}{2\theta^2} = (n-1).  \]
\end{enumerate}
For all the test matrices, the constants $\Delta_1$ and $\Delta_2$ increase with increasing $\theta$ as the offdiagonal elements become larger in magnitude relative to the diagonal elements. Therefore, we expect the Rademacher Monte Carlo estimators to lose accuracy with increasing $\theta$, as measured by
 the normwise relative error (NRE) in the computed diagonal $\D{\widehat{\M{A}}}$,
\begin{equation*}
\mathrm{NRE} \equiv \frac{\|\D{\M{A}} - \D{\widehat{\M{A}}}\|_2}{\| \D{\M{A}}\|_2},
\end{equation*}
in Figures \ref{fig:testmat1}-\ref{fig:sparserad}.

\begin{figure}[!ht]
    \centering
    \includegraphics[scale=0.3]{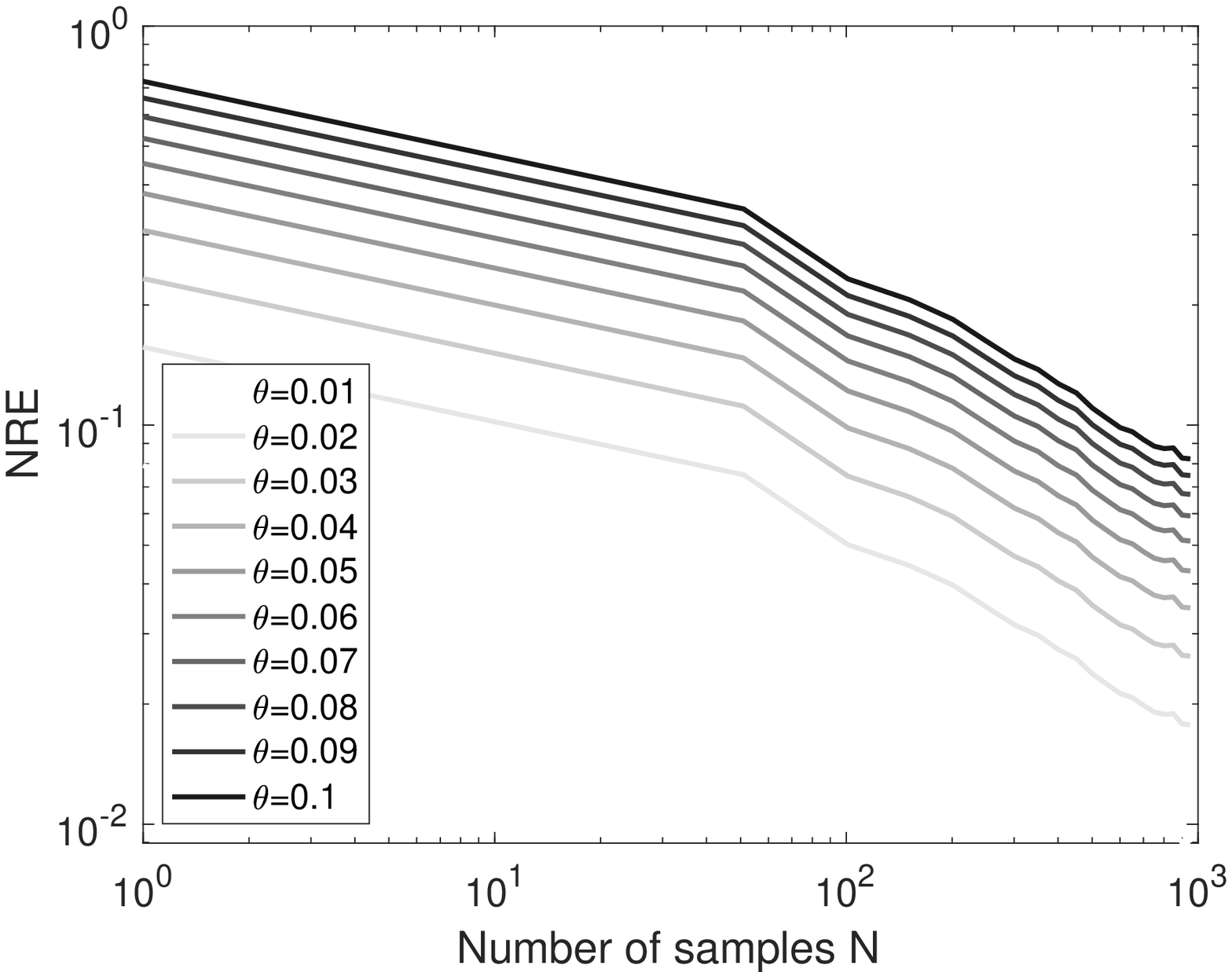}
    \includegraphics[scale=0.35]{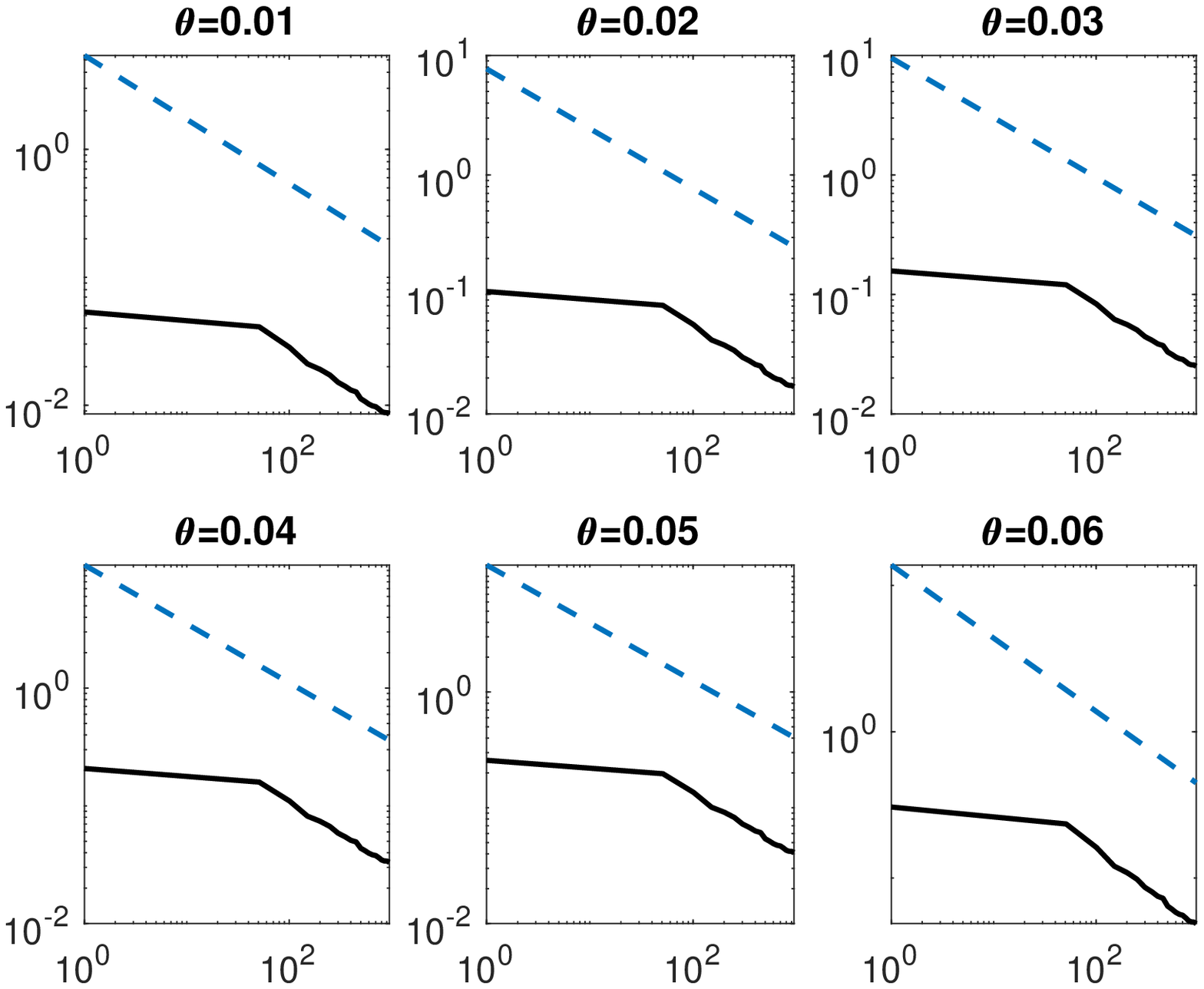}
        \caption{Rademacher Monte Carlo estimator applied to Test Matrix 1. 
        Big left panel: NRE for different values of $\theta$ versus
        sampling amount $N$. 
Small panels on the right: NRE (solid black line), and bound (\ref{eqn:c33})
(blue dotted line) versus sampling amount~$N$ with failure
probability $\delta=10^{-16}$.}
    \label{fig:testmat1}
\end{figure}

\subsection{Experiment 1: Accuracy of Rademacher Monte Carlo estimator on test matrices}\label{s_exp1}
Figures \ref{fig:testmat1}-\ref{fig:testmat3} show the NRE of the
Rademacher Monte Carlo estimator applied to the  test matrices in 
section~\ref{s_test}, and the bounds from the normwise
$(\epsilon,\delta)$ estimators in \cref{c:rademacher}.  

The big left panel
displays the NRE versus the sampling amount $N$. 
This NRE represents the average of the NREs over $10$ different independent runs.
The small panels on the right show the bound $\epsilon$ for the normwise
$(\epsilon,\delta)$ estimators from Corollary~\ref{c:rademacher} with failure probability $\delta = 10^{-16}$.

For Corollary~\ref{c:rademacher}, we solve for $\epsilon$ from the simpler bound
\begin{equation*} N \geq \frac{\Delta_2}{3\epsilon^2} \left( 2 + 6 \Delta_3\right) \ln (8d/\delta),
\qquad \Delta_3\equiv \frac{\Delta_1}{\Delta_2},
\end{equation*}
to obtain
\begin{equation}\label{eqn:c33}
\epsilon = \sqrt{\frac{\Delta_2}{3N} \left( 2 + 6 \Delta_3\right) \ln (8d/\delta)},\qquad 
\Delta_2=\frac{\|\M{A}-\D{\M{A}}\|_{\infty}}{\|\D{\M{A}}\|_{\infty}},\quad
\Delta_3=\frac{K_1}{\|\M{A}-\D{\M{A}}\|_{\infty}}.
\end{equation}

The big left panels illustrate that, for a fixed sampling amount $N$,
the NRE for Test Matrices 1 and 3 increases with $\theta$. This is 
because the offdiagonals become more dominant as $\theta$ becomes larger.

\begin{figure}[!ht]
    \centering
    \includegraphics[scale=0.3]{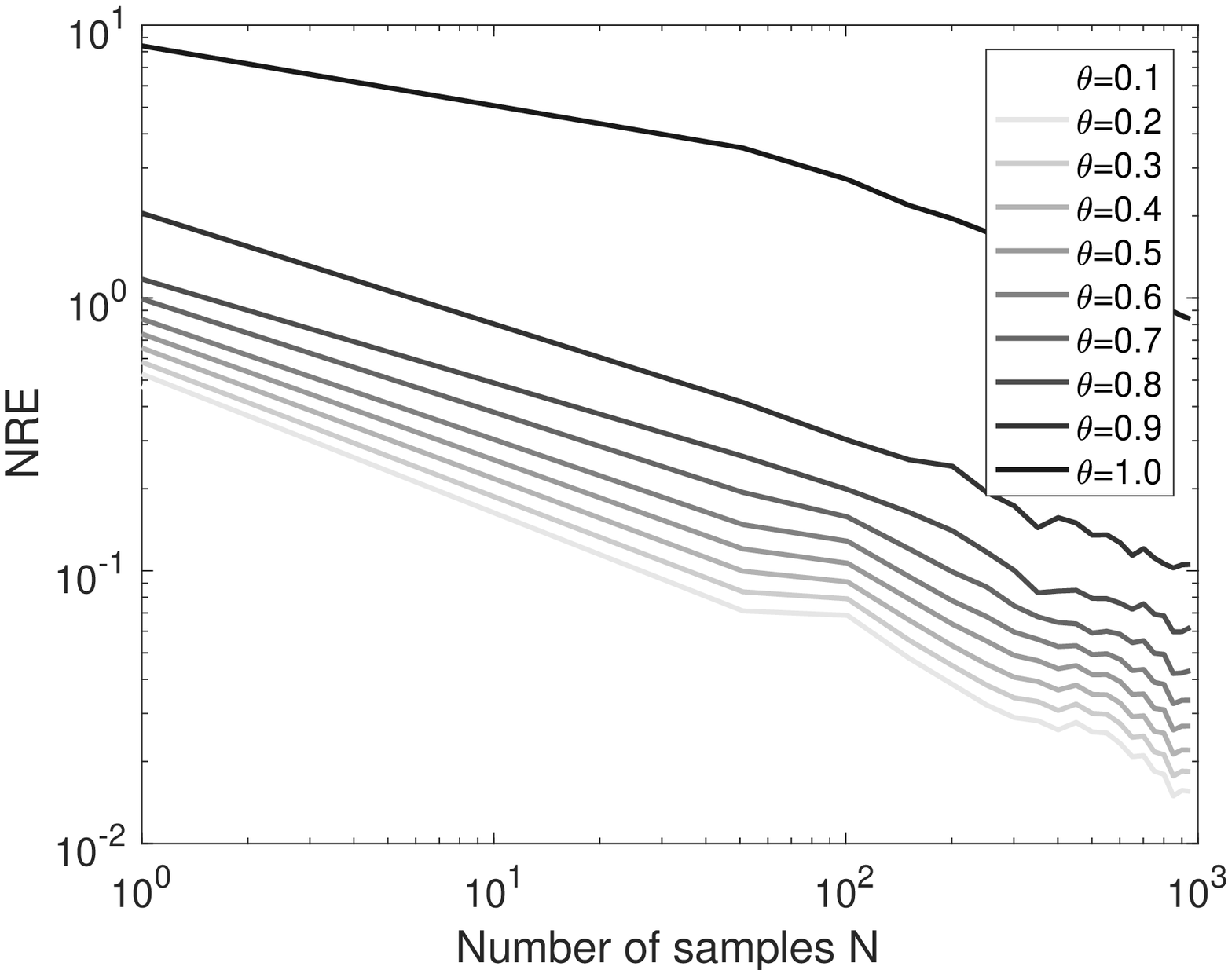}
    \includegraphics[scale=0.35]{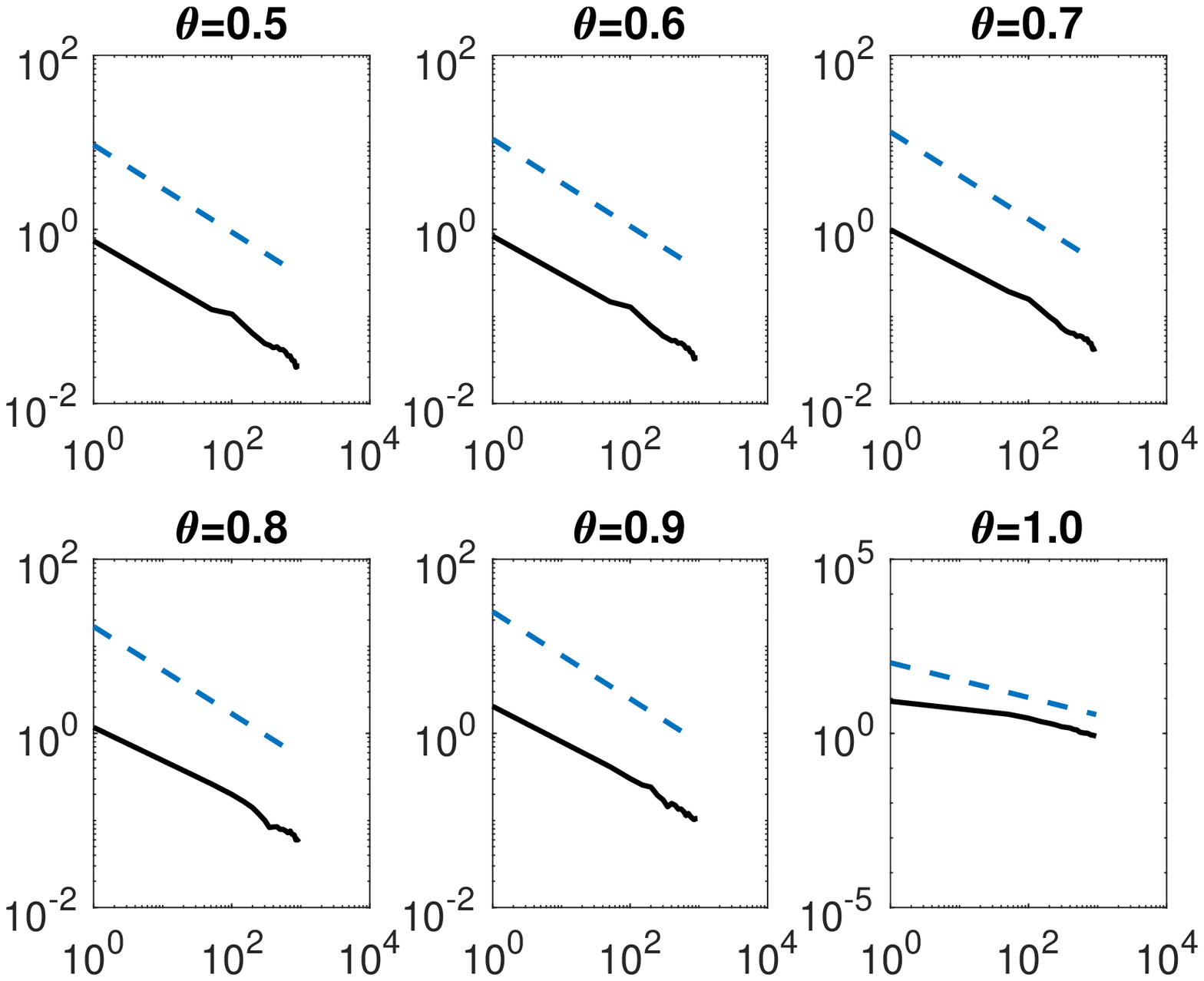}
 \caption{Rademacher Monte Carlo estimator applied to Test Matrix 2. 
        Big left panel: NRE for different values of $\theta$ versus 
        sampling amount $N$. 
        Small panels on the right: NRE (solid black line) and
bound (\ref{eqn:c33}) (blue dotted line)  versus sampling amount~$N$ with failure
probability $\delta=10^{-16}$.}
\label{fig:testmat2}
\end{figure}

\begin{figure}[!ht]
    \centering
    \includegraphics[scale=0.3]{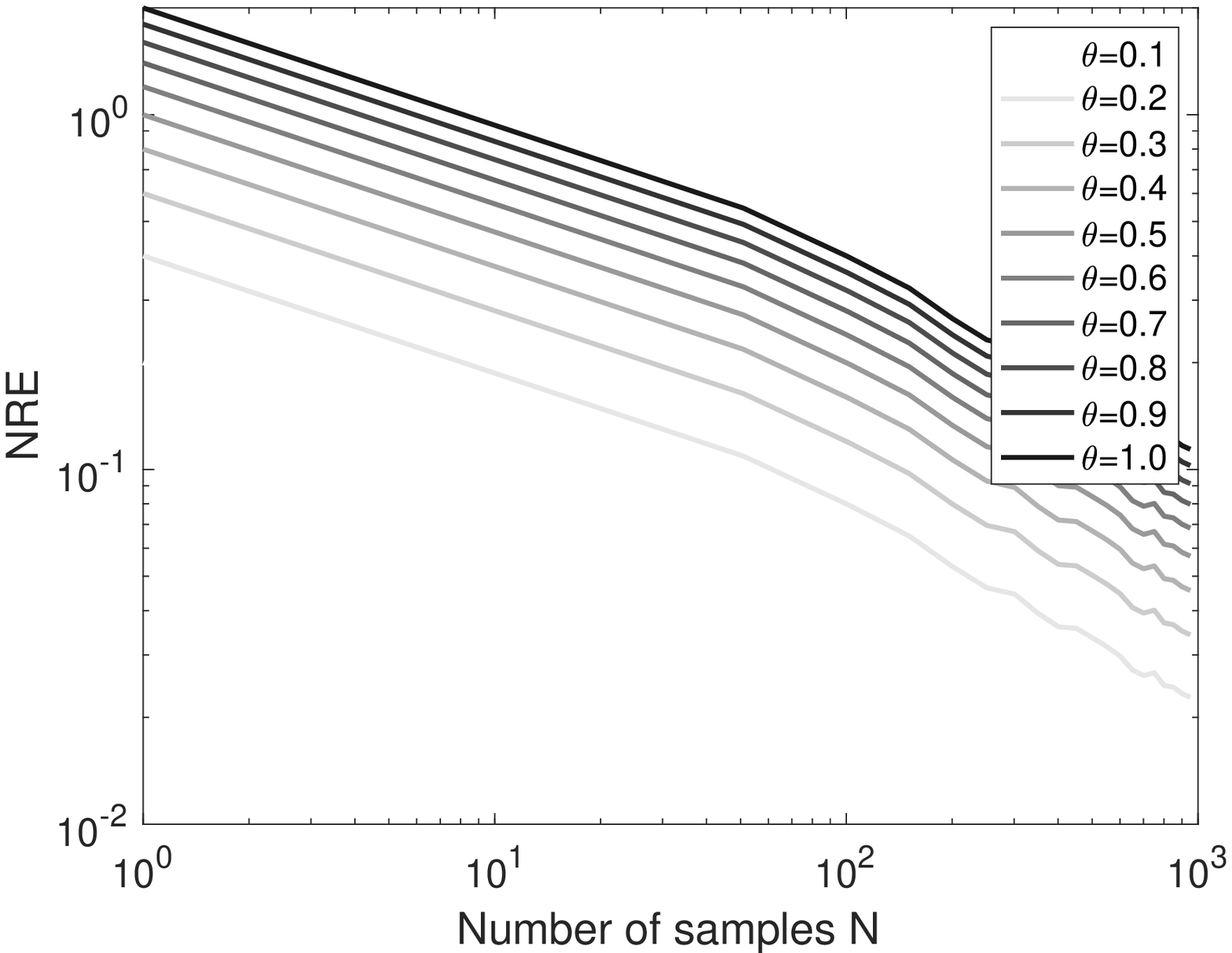}
    \includegraphics[scale=0.35]{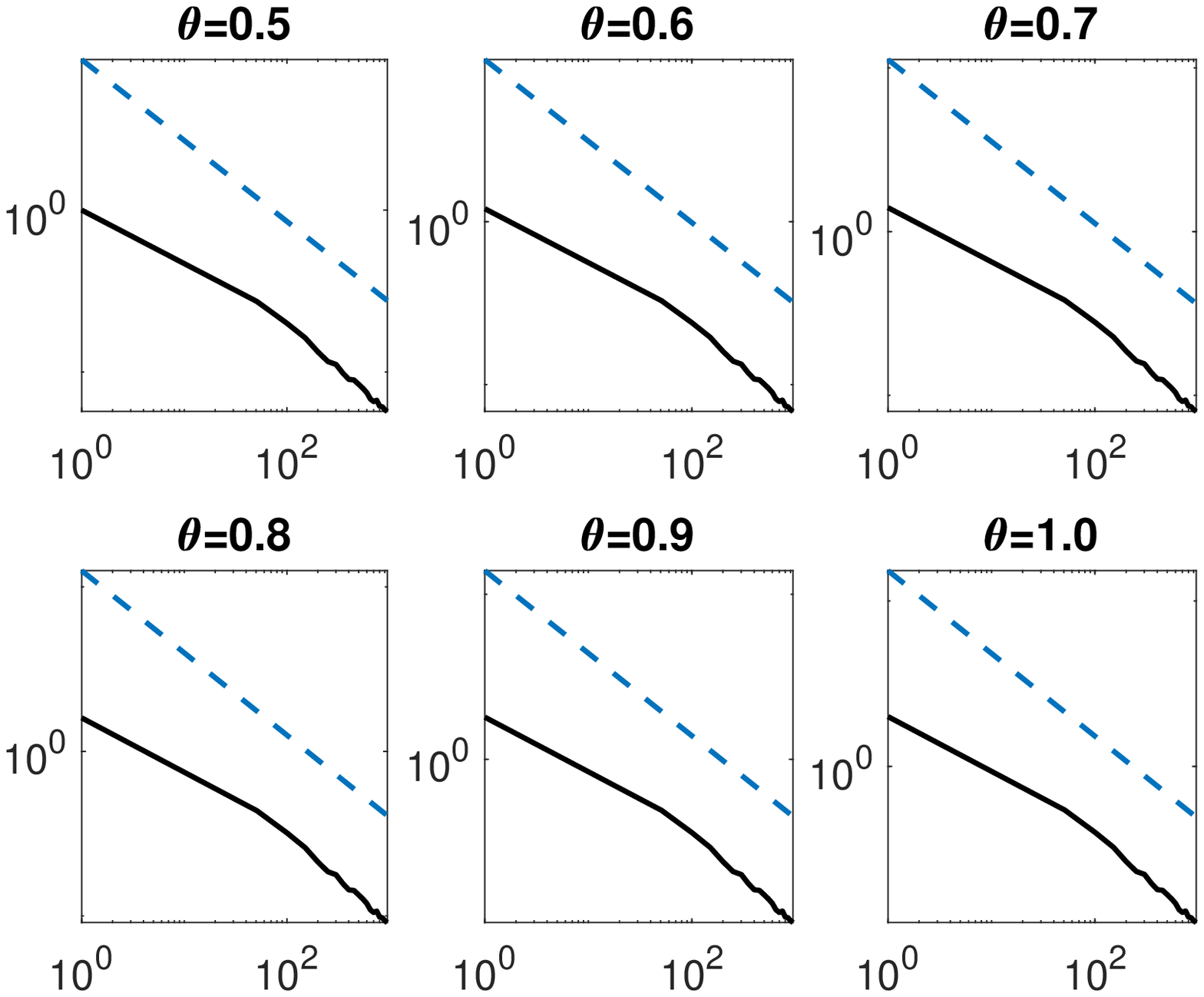}
 \caption{Rademacher Monte Carlo estimator applied to Test Matrix 3. 
        Big left panel: NRE for different values of $\theta$ versus
        sampling amount $N$. 
Small panels on the right: NRE (solid black line) and
bound (\ref{eqn:c33}) (blue dotted line)  versus sampling amount~$N$ with failure
probability $\delta=10^{-16}$.}
 \label{fig:testmat3}
\end{figure}

\subsection{Experiment 2: Different Monte Carlo estimators}\label{e_exp2}
We compare the accuracy of the following Monte Carlo estimators on 
Test Matrix 1 with $\theta = 0.01$:
Rademacher, Gaussian, sparse Rademacher with $s=3$, and normalized Gaussian.

For each estimator,
Figure~\ref{fig:compmethods} shows the mean of the NRE and variance over 100 runs,
with the shaded regions representing the $2.5\%$ and $97.5\%$ quantiles.

The normalized Gaussian estimator is about as
accurate as the Rademacher estimator, while the sparse Rademacher with 
$s=3$ is about as accurate as the Gaussian estimator. The
Gaussian and sparse Rademacher estimators are less accurate
than the Rademacher and normalized Gaussian estimators.
The shaded regions illustrate that, as expected,
the sample variance of all estimators decreases with increasing sampling 
amount $N$.

\begin{figure}[!ht]
    \centering
    \includegraphics[scale=0.3]{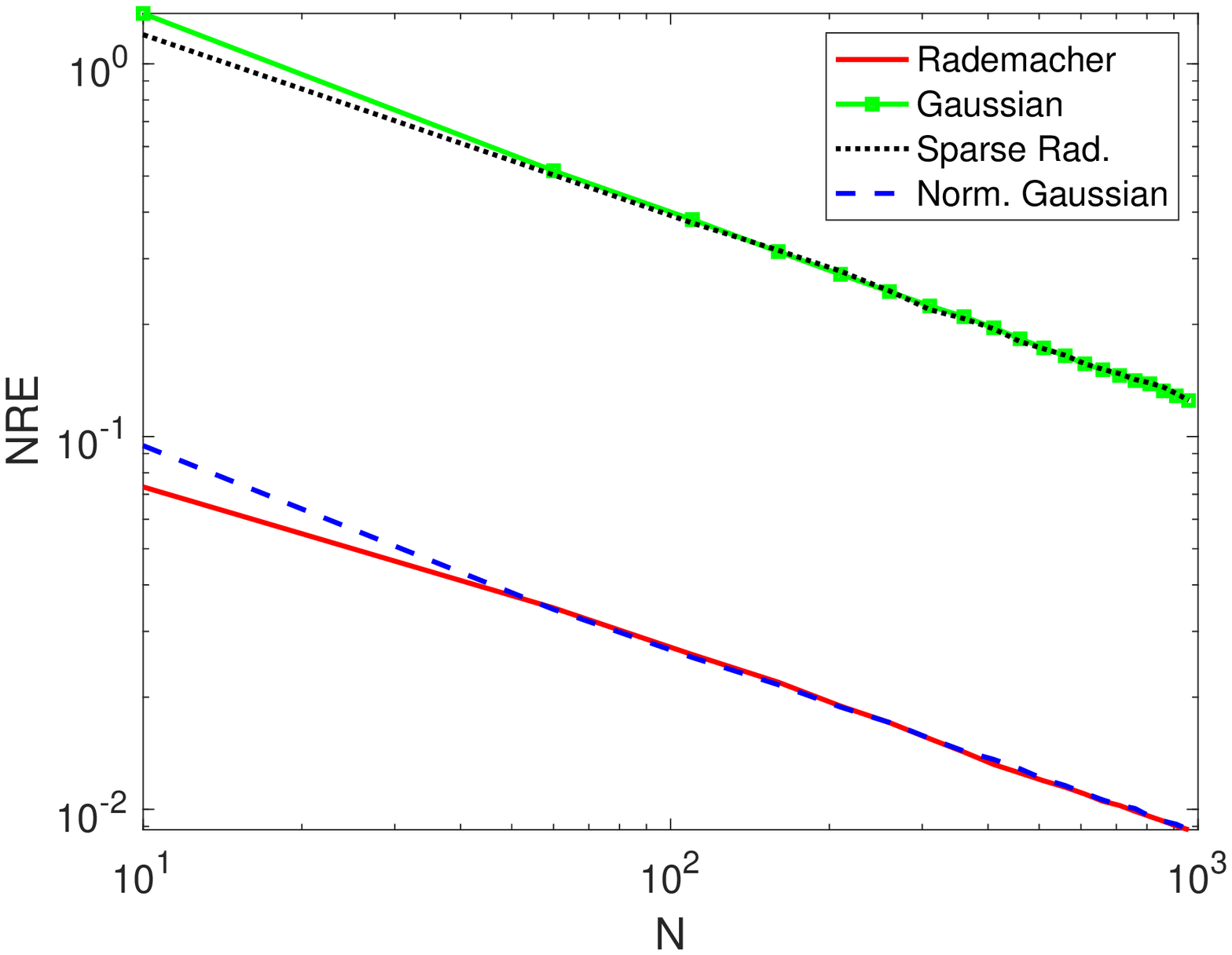}
    \includegraphics[scale=0.35]{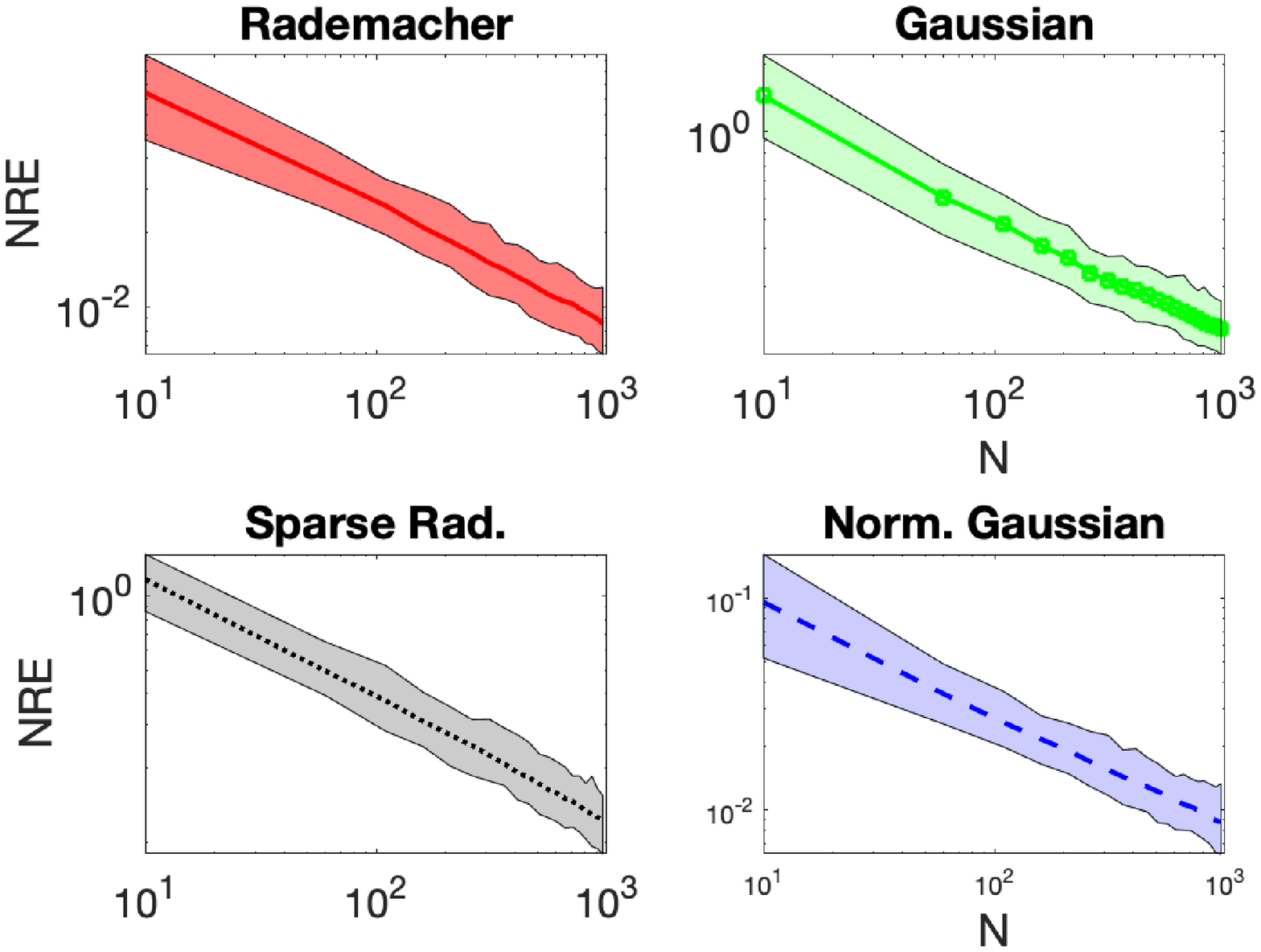}
    \caption{Rademacher, Gaussian, sparse Rademacher with $s=3$, and Normalized Gaussian Monte Carlo estimators applied to Test Matrix 1 with $\theta = 0.01$. Big left panel: NRE mean versus
        sampling amount $N$ for different estimators. Small right panels: 
NRE mean (styled lines), and $2.5\%$ and $97.5\%$ quantiles (shaded regions)
versus sampling amount $N$.}
    \label{fig:compmethods}
\end{figure}

\subsection{Experiment 3: Effect of sparsity in Rademacher vectors}\label{s_exp3}
We apply the Rademacher Monte Carlo estimator to Test Matrix 1 with $\theta=.01$
with four different sparsity levels: 
$s=1$ (standard Rademacher), $s=3$ \cite{achlioptas2003database}, $s=10$, and $s=50$.

For each sampling amount $N$,
Figure~\ref{fig:sparserad} shows the mean and the variance of the NRE
over 100 runs. It suggests that
sparse Rademacher estimators ($s>1$) may not be able to achieve a single
digit of accuracy, unless the sampling amount is so large as to exceed the matrix dimension.

\begin{figure}[!ht]
    \centering
    \includegraphics[scale=0.4]{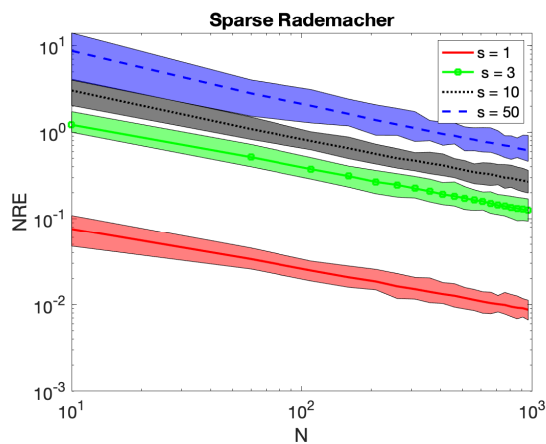}
    \caption{Sparse Rademacher Monte Carlo estimators with sparsity levels 
    $s=1, 3, 10, 50$ applied to Test Matrix 1 with $\theta = 0.01$. 
    NRE (dotted lines) and $2.5\%$ and $97.5\%$ quantiles (shaded regions).}
    \label{fig:sparserad}
\end{figure}

\subsection{Example 4: Bounds for DGSM Monte Carlo estimator}\label{s_exp4} We apply the DGSM
Monte Carlo estimator (\ref{e_6MC}) to the diagonal matrix
\begin{equation}\label{e_6diag}
\M{S}\equiv\diag(\V{s}) \in \R^{n\times n}\qquad
s_j \equiv  \exp(-10j/n), \qquad 1 \leq j \leq n
\end{equation}
from the quadratic function in Section~\ref{s_6ex} for $n=100$, and 
illustrate the accuracy of Corollary~\ref{c_66}.

The left panel of Figure~\ref{fig:dgsmquad} shows the normwise relative error 
\begin{equation*}
\mathrm{NRE}\equiv\frac{\|\D{\M{C}} - \D{\M{\widehat{C}}}\|_2}{\|\D{\M{C}}\|_2} 
\end{equation*}
which represents the the average of the NREs over $100$ independent runs.

For Corollary~\ref{c_66}, we fix the sample size $N$ and solve for $\epsilon$ from the simpler bound
\[ N \geq \frac{S_2}{3\epsilon^2} \left( 2 + 6S_3\right) \ln (8d/\delta),
\qquad S_3\equiv \frac{S_1}{c_{\max}S_2},
\]
to obtain
\begin{equation}\label{eqn:cex4}
\epsilon = \sqrt{\frac{S_2}{3N} \left( 2 + 6 S_3\right) \ln (8d/\delta)}.
\end{equation}
The expressions for $S_1,S_2,c_{\max}$ and $d$ for this example have been derived in subsection~\ref{s_6ex}.

\begin{figure}[!ht]
    \centering
    \includegraphics[scale=0.5]{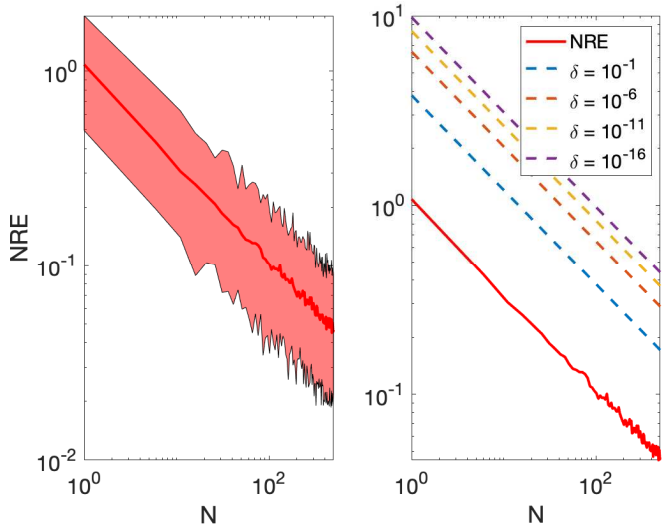}
    \caption{DGSM Monte Carlo estimator (\ref{e_6MC}) applied to $100\times 100$ matrix $\M{S}$ in (\ref{e_6diag}). Left panel: NRE mean (solid line) and  $2.5\%$ and $97.5\%$ quantiles (shaded regions)
    versus sampling amount $N$. Right panel: NRE and bounds
(\ref{eqn:cex4}) for different failure probabilities $\delta$    
    versus sampling amount.}
    \label{fig:dgsmquad}
\end{figure}

The right panel of Figure~\ref{fig:dgsmquad} illustrates that
with less stringent failure probabilities $\delta$, the relative
bounds (\ref{eqn:cex4}) move closer to the NRE.

\subsection{Example 5: DGSM on the Circuit model}
We apply the Monte Carlo DGSM estimator (\ref{e_6MC}) to 
the so-called \textit{circuit model}
from~\cite{constantine2017global}. The  quantity of interest being modeled is the midpoint voltage of a  transformerless push-pull circuit, which depends on $n=6$ parameters through a nonlinear closed-form algebraic expression.

As in~\cite{constantine2017global}, we normalize the parameter space to $\mathcal{X} = [-1, 1]^n$ and scale the partial derivatives 
appropriately\footnote{MATLAB codes are available in \url{https://bitbucket.org/paulcon/global-sensitivity-metrics-from-active-subspaces/src/master/}.}.

\begin{figure}[!ht]
    \centering
    \includegraphics[scale=0.4]{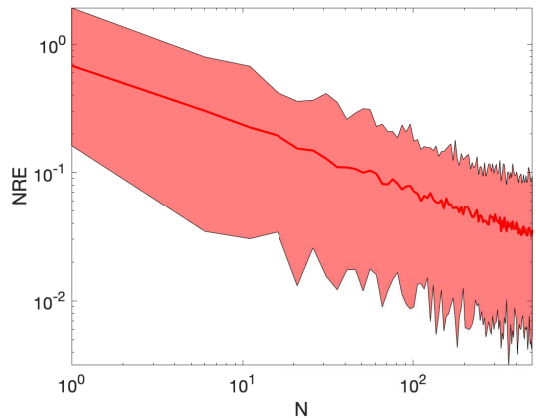}
    \caption{DGSM Monte Carlo estimator (\ref{e_6MC}) applied to the Circuit model. NRE mean (solid line) and  $2.5\%$ and $97.5\%$ quantiles (shaded regions) versus sampling amount $N$.}
    \label{fig:my_label}
\end{figure}

Figure~\ref{fig:my_label} shows the mean 
of the NRE and variances over $100$
independent runs per sampling amount $N$,
with the shaded regions representing  $2.5\%$ and $97.5\%$ quantiles. Since the exact expressions for the DGSMs are unavailable, we use as the exact value 
a tensor product Gauss-Legendre quadrature-based approximation with $15$ points per dimension (i.e., $15^n$ total points).

\section{Conclusion and future work}
This paper derives probabilistic bounds for the Monte Carlo diagonal estimators; the bounds come in two flavors, normwise and componentwise bounds for the absolute and relative errors of the Monte Carlo estimators. There are several avenues for future work. First, it would be interesting to study the accuracy of the diagonal estimator for matrix functions when a polynomial, or rational approximation to the matrix function is used. Second, we are currently pursuing the extension of the analysis of the Monte Carlo diagonal estimators to estimators for the  selected entries (possibly including offdiagonal entries) of a matrix. 

\section{Acknowledgements}
The authors would like to acknowledge support from the National Science Foundation through the grants DMS-1745654 (all three authors) and DMS-1845406 (A.K.S.). We would like to acknowledge Alen Alexanderian for helpful discussions.

\bibliography{refs}

\begin{thebibliography}{10}

\bibitem{achlioptas2003database}
D.~Achlioptas.
\newblock Database-friendly random projections: {J}ohnson-{L}indenstrauss with
  binary coins.
\newblock {\em J. Comput. System Sci.}, 66(4):671--687, 2003.

\bibitem{avron2011randomized}
H.~Avron and S.~Toledo.
\newblock Randomized algorithms for estimating the trace of an implicit
  symmetric positive semi-definite matrix.
\newblock {\em J. ACM}, 58(2):1--34, 2011.

\bibitem{BN22}
R.~A. Baston and Y.~Nakatsukasa.
\newblock Stochastic diagonal estimation: probabilistic bounds and an improved
  algorithm, 2022.
\newblock arXiv:2201.10684.

\bibitem{bekas2007estimator}
C.~Bekas, E.~Kokiopoulou, and Y.~Saad.
\newblock An estimator for the diagonal of a matrix.
\newblock {\em Appl. Numer. Math.}, 57(11-12):1214--1229, 2007.

\bibitem{chen2012masked}
R.~Y. Chen, A.~Gittens, and J.~A. Tropp.
\newblock The masked sample covariance estimator: an analysis using matrix
  concentration inequalities.
\newblock {\em Inf. Inference}, 1(1):2--20, 2012.

\bibitem{constantine2017global}
P.~G. Constantine and P.~Diaz.
\newblock Global sensitivity metrics from active subspaces.
\newblock {\em Reliab. Eng. Syst. Safe.}, 162:1--13, 2017.

\bibitem{cortinovis2021randomized}
A.~Cortinovis and D.~Kressner.
\newblock On randomized trace estimates for indefinite matrices with an
  application to determinants.
\newblock {\em Found. Comput. Math.}, pages 1--29, 2021.

\bibitem{hutchinson1989stochastic}
M.~F. Hutchinson.
\newblock A stochastic estimator of the trace of the influence matrix for
  {L}aplacian smoothing splines.
\newblock {\em Comm. Statist. Simulation Comput.}, 18(3):1059--1076, 1989.

\bibitem{kaperick2019diagonal}
B.~J. Kaperick.
\newblock Diagonal estimation with probing methods.
\newblock Master's thesis, Virginia Polytechnic Institute and State University,
  2019.

\bibitem{kucherenko2017}
S.~Kucherenko and B.~Iooss.
\newblock {\em Derivative-Based Global Sensitivity Measures}, pages 1241--1263.
\newblock Springer International Publishing, Cham, 2017.

\bibitem{laeuchli2016methods}
J.~H. Laeuchli.
\newblock {\em Methods for Estimating The Diagonal of Matrix Functions}.
\newblock PhD thesis, College of William and Mary, 2016.

\bibitem{lam2020multifidelity}
R.~R. Lam, O.~Zahm, Y.~M. Marzouk, and K.~E. Willcox.
\newblock Multifidelity dimension reduction via active subspaces.
\newblock {\em SIAM J. Sci. Comput.}, 42(2):A929--A956, 2020.

\bibitem{Larsen2006}
J.~L. Larsen and L.~M. Morris.
\newblock {\em An Introduction to Mathematical Statistics and its
  Applications}.
\newblock Pearson Prentice Hall, fourth edition, 2006.

\bibitem{LHC06}
P.~Li, T.~J. Hastie, and K.~W. Church.
\newblock Very sparse random projections.
\newblock In {\em Proceedings of the 12th ACM SIGKDD International Conference
  on Knowledge Discovery and Data Mining}, KDD '06, page 287–296, New York,
  NY, USA, 2006. Association for Computing Machinery.

\bibitem{Mitz2005}
M.~Mitzenmacher and E.~Upfal.
\newblock {\em Probability and computing}.
\newblock Cambridge University Press, Cambridge, 2005.
\newblock Randomized algorithms and probabilistic analysis.

\bibitem{roosta2015improved}
F.~Roosta-Khorasani and U.~Ascher.
\newblock Improved bounds on sample size for implicit matrix trace estimators.
\newblock {\em Found. Comput. Math.}, 15(5):1187--1212, 2015.

\bibitem{soms1976asymptotic}
A.~P. Soms.
\newblock An asymptotic expansion for the tail area of the t-distribution.
\newblock {\em J. Amer. Statist. Assoc.}, 71(355):728--730, 1976.

\bibitem{tropp2015introduction}
J.~A. Tropp.
\newblock An introduction to matrix concentration inequalities.
\newblock {\em Found. Trends Mach. Learning}, 8(1--2):1--230, 2015.

\bibitem{ubaru2018applications}
S.~Ubaru and Y.~Saad.
\newblock Applications of trace estimation techniques.
\newblock In T.~Kozubek, M.~{\v{C}}erm{\'a}k, P.~Tich{\'y}, R.~Blaheta,
  J.~{\v{S}}{\'i}stek, D.~Luk{\'a}{\v{s}}, and J.~Jaro{\v{s}}, editors, {\em
  High Performance Computing in Science and Engineering}, pages 19--33, Cham,
  2018. Springer International Publishing.

\bibitem{vershynin2018high}
R.~Vershynin.
\newblock {\em High-dimensional probability: {A}n introduction with
  applications in data science}, volume~47.
\newblock Cambridge University Press, 2018.

\bibitem{wathen2015preconditioning}
A.~J. Wathen.
\newblock Preconditioning.
\newblock {\em Acta Numer.}, 24:329--376, 2015.

\bibitem{yao2020adahessian}
Z.~Yao, A.~Gholami, S.~Shen, M.~Mustafa, K.~Keutzer, and M.~W. Mahoney.
\newblock {ADAHESSIAN}: {A}n adaptive second order optimizer for machine
  learning.
\newblock {\em arXiv preprint arXiv:2006.00719}, 2020.

\end{thebibliography}
\bibliographystyle{abbrv}
\end{document}